\documentclass[12pt]{article}
\usepackage{graphicx}
\usepackage{color,latexsym,amsfonts,amssymb}
\usepackage{bm}     %use \bm instead of \mathbf
\usepackage{amsthm} %command for \proof
\usepackage{amsmath}
\usepackage{booktabs} %table for 3 lines, toprule/midrule/bottomrule
\usepackage{boxedminipage}  %the box style for equations
\usepackage{natbib} %\citep
\usepackage{algorithm,algpseudocode}  %for pseudo code writing
\usepackage{subfigure}

\DeclareMathOperator*{\argmax}{argmax}
\DeclareMathOperator*{\argmin}{argmin}

\newtheorem{theorem}{Theorem}
\newtheorem{corollary}{Corollary}
\newtheorem{lemma}{Lemma}

\newcommand{\CVaR}{\mathrm{CVaR}}

\newcommand{\T}{\mathrm{T}}
\vspace{-1.2cm}

\usepackage{epstopdf}

\begin{document}

\title{\vspace{-1cm} Global Algorithms for Mean-Variance Optimization in Markov Decision Processes}
\author{Li Xia, Shuai Ma \thanks{L. Xia and S. Ma are both with the School of Business, Sun Yat-Sen University, Guangzhou 510275, China. (email: xiali5@sysu.edu.cn).}}
\date{}

\maketitle

\vspace{-1cm}

\begin{abstract}
Dynamic optimization of mean and variance in Markov decision
processes (MDPs) is a long-standing challenge caused by the failure
of dynamic programming. In this paper, we propose a new approach to
find the globally optimal policy for combined metrics of
steady-state mean and variance in an infinite-horizon undiscounted
MDP. By introducing the concepts of pseudo mean and pseudo variance,
we convert the original problem to a bilevel MDP problem, where the
inner one is a standard MDP optimizing pseudo mean-variance and the
outer one is a single parameter selection problem optimizing pseudo
mean. We use the sensitivity analysis of MDPs to derive the
properties of this bilevel problem. By solving inner standard MDPs
for pseudo mean-variance optimization, we can identify worse policy
spaces dominated by optimal policies of the pseudo problems. We
propose an optimization algorithm which can find the globally
optimal policy by repeatedly removing worse policy spaces. The
convergence and complexity of the algorithm are studied. Another
policy dominance property is also proposed to further improve the
algorithm efficiency. Numerical experiments demonstrate the
performance and efficiency of our algorithms. To the best of our
knowledge, our algorithm is the first that efficiently finds the
globally optimal policy of mean-variance optimization in MDPs. These
results are also valid for solely minimizing the variance metrics in
MDPs.
\end{abstract}


\section{Introduction}\label{section_intro}
Mean-variance optimization is an important model for the risk
control in finance engineering, which was first proposed by
\cite{Markowitz52} for single-period portfolio management. Extending
to multi-period scenarios is a natural but challenging research
topic. This is because the variance criterion in multi-period is not
additive, which induces the time inconsistency and the failure of
dynamic programming. This important topic attracts research
attention over past decades
\citep{Dai21,Gao13,Hernandez99,Sobel94,Sobel82}, while it is not
completely solved yet.

Since Markov models are widely used to study multi-period stochastic
systems, there is rich literature on Markov decision processes
(MDPs) with variance related criteria, either for discounted or
undiscounted, discrete-time or continuous-time, discrete-state or
continuous-state, finite-horizon or infinite-horizon MDPs. Excellent
works can be referred to
\cite{Chung94,Filar85,Haskell13,Hernandez99,Sobel82,Sobel94,Guo09b},
just to name a few. Many of these works study the variance
minimization of accumulated rewards in a policy set, in which the
mean performance has already been optimized. In such scenarios, the
variance minimization problem can be equivalently converted to
another standard MDP with a new cost function
\citep{Guo12,Huang18,Sobel82,Xia18}. These approaches are not
applicable to directly optimize variance or mean-variance combined
metrics in MDPs when the mean performance is not optimized. Another
method to study the mean-variance optimization of MDPs is to
reformulate these problems as mathematical programming models and to
do further analytical investigations
\citep{Chung94,Haskell13,Sobel94}. How to efficiently solve these
mathematical programs is challenging.

Another research stream on multi-period mean-variance optimization
is from the perspective of stochastic control. The seminal work by
\cite{Li00,Zhou00} formulated the mean-variance portfolio selection
problem as a linear quadratic (LQ) control problem and used an
embedding method to develop an iterative procedure to analytically
solve this problem. There are numerous works following this research
line \citep{Gao13,Zhou04,Zhu04} and interested audience can refer to
a recent survey paper \citep{Cui22}. However, these works use an LQ
model with linear state transitions, which may properly characterize
the portfolio selection problem but lack much generalization
compared with Markov models.

Recently there are also some works that study mean-variance
optimization in the regime of reinforcement learning. Although the
principle of dynamic programming fails, gradient-based algorithms
for parameterized policies (represented by neural networks) still
work. Most of these studies focus on improving the sampling
efficiency for learning the gradient estimators for variance related
metrics \citep{Borkar10,Prashanth13,Tamar12}. A recent progress is
to reformulate mean-variance optimization with Fenchel duality
\citep{Xie18}, and to adopt gradient-based algorithms to find local
optima \citep{Bisi20,Zhang21}. However, all these gradient-based
learning algorithms suffer from slow convergence speed and trap into
local optima. Globally solving the mean-variance optimization
problem in MDPs is still an unanswered question.

In this paper, we study global algorithms for the mean-variance
optimization problem in an infinite-horizon discrete-time
undiscounted MDP. The mean and variance of rewards are measured in a
steady-state environment, similar to those in the works by
\cite{Bisi20,Chung94,Sobel94,Xia16a}. By introducing an auxiliary
variable called pseudo mean $y \in \mathbb R$, we convert the
steady-state mean-variance optimization problem to a bilevel MDP
problem, where the inner level is a standard MDP $\mathcal M(y)$
optimizing the so-called pseudo mean-variance and the outer level is
a single parameter selection problem optimizing the pseudo mean $y$.
With the sensitivity analysis of MDPs, we show that the optimal
value of the pseudo mean-variance optimization problem $\mathcal
M(y)$ is a convex piecewise quadratic function with respect to $y$
and its global optimum equals the optimum of the mean-variance
optimization problem. We further discover policy dominance
properties which help us discard the worse policies dominated by the
optimal policy of $\mathcal M(y)$. Thus, the optimization complexity
can be significantly reduced. Based on these properties, we develop
an iterative algorithm which is shown to find the global optimum of
the mean-variance optimization problem after a finite number of
iterations. The computation complexity and some variants of the
algorithm are also studied. Compared with the literature work only
capable of finding a local optimum of mean-variance optimization in
MDPs \citep{Xia20}, our algorithms guarantee a global convergence.
The performance and efficiency of our algorithms are also
demonstrated by numerical experiments. To the best of our knowledge,
our work is the first to compute the globally optimal policies of
mean-variance optimization in MDPs.

The rest of the paper is organized as follows. In
Section~\ref{section_model}, we give the MDP formulation for the
mean-variance optimization problem. Section~\ref{section_result}
presents the main results of this paper, including the policy
dominance property and the algorithmic analysis. Numerical
experiments are conducted in Section~\ref{section_experiment} to
demonstrate the performance of our algorithms. Finally, we conclude
this paper in Section~\ref{section_conclusion}.

\section{Problem Formulation}\label{section_model}
Consider an infinite-horizon discrete-time MDP denoted by a tuple
$\mathcal M = \langle \mathcal S, \mathcal A, \mathcal P, \bm r
\rangle$, where $\mathcal S=\{1,2,\dots,S\}$ is the state space,
$\mathcal A=\{a_1,a_2,\dots,a_A\}$ is the action space, $\mathcal P
: \mathcal S \times \mathcal A \overset{D}{\mapsto} \mathcal S$ is
the state transition probability kernel with element $p(j|i,a)$
where $\overset{D}{\mapsto}$ represents a mapping to the
distribution on $\mathcal S$, and $\bm r : \mathcal S \times
\mathcal A \mapsto \mathbb R$ is the reward function with element
$r(i,a)$, $i,j \in \mathcal S$, $a \in \mathcal A$. When the system
is in state $i$ and action $a$ is adopted, it will transit to the
next state $j$ with probability $p(j|i,a)$ and a reward $r(i,a)$ is
incurred. Since deterministic policies can attain optimal mean and
variance in MDPs \citep{Haskell13,Xia20}, we only consider
stationary deterministic policies $d: \mathcal S \mapsto \mathcal
A$, where $d(i) \in \mathcal A$ indicates the action adopted in
state $i$. The corresponding policy space is denoted by $\mathcal D$
and we assume that the MDP with any policy $d \in \mathcal D$ is a
unichain. When a policy $d$ is adopted, the state transition
probability matrix is denoted by $\bm P^d$ and its $(i,j)$-th
element is $p(j|i,d(i))$, $i,j \in \mathcal S$. The associated
steady-state distribution is denoted by an $S$-dimensional row
vector $\bm \pi^d:=(\pi^d(i))_{i \in \mathcal S}$. Obviously, we
have
\begin{equation*}
\bm \pi^d \bm P^d = \bm \pi^d,  \quad  \bm \pi^d \bm e = 1, \quad
\bm P^d \bm e = \bm e,
\end{equation*}
where $\bm e$ is a column vector of 1's with a proper dimension
size. We consider long-run performance metrics of this MDP, which
are independent of the initial state at time 0. The \emph{long-run
average (mean)} reward of the MDP under policy $d$ is defined as
\begin{equation}
\mu^d := \lim\limits_{T \rightarrow \infty} \mathbb E^d
\left\{\frac{1}{T}\sum_{t=0}^{T-1}r(X_t, A_t) \right\} = \bm \pi^d
\bm r^d,
\end{equation}
where $\mathbb E^d$ indicates the expectation under policy $d$,
$X_t$ is the system state at time $t$, $A_t = d(X_t)$ is the action
adopted at time $t$, $\bm r^d$ is an $S$-dimensional column vector
whose element is $r(i,d(i))$, $i \in \mathcal S$. Similarly, the
\emph{long-run variance} (or steady-state variance) of the MDP under
policy $d$ is defined as \citep{Xia16a,Xia20}
\begin{equation}
\sigma^d := \lim\limits_{T \rightarrow \infty}\mathbb E^d
\left\{\frac{1}{T}\sum_{t=0}^{T-1}[r(X_t, A_t) - \mu^d]^2 \right\} =
\bm \pi^d (\bm r^d - \mu^d \bm e)_{\odot}^2,
\end{equation}
where $(\bm r^d - \mu^d \bm e)_{\odot}^2$ is the component-wise
square of vector $(\bm r^d - \mu^d \bm e)$, i.e.,
\begin{equation*}
(\bm r^d - \mu^d \bm e)_{\odot}^2 := ((r(1,d(1)) - \mu^d)^2, \
(r(2,d(2)) - \mu^d)^2, \ \dots, \ (r(S,d(S)) - \mu^d)^2 )^\T.
\end{equation*}
When the finite Markov chain is a unichain, %ergodic (irreducible and aperiodic)
we can view $r(X_t,A_t)$ as a random variable whose value
realization set is $\{r(i,d(i)): i\in \mathcal S\}$ and distribution
is $(\bm P^d)^t \bm \nu$, where $\bm \nu$ is the vector of initial
state distribution. We can verify that
\begin{eqnarray*}
\mu^d &=& \lim\limits_{t \rightarrow \infty} \mathbb E[r(X_t,A_t)], \\
\sigma^d &=& \lim\limits_{t \rightarrow \infty} \mathbb
{VAR}[r(X_t,A_t)].
\end{eqnarray*}

Mean-variance optimization was originally proposed by
\cite{Markowitz52} for portfolio selection, where decision makers
usually aim at maximizing the mean return while minimizing the
variance risk, which is a multi-objective optimization problem.
Usually, the Pareto frontier composed of Pareto efficient solutions
is the optimization goal, which is illustrated by
Fig.~\ref{fig_pareto}. A common way of obtaining Pareto optima is to
optimize the combined objective
\begin{equation}\label{eq_eta}
\eta^d := \beta \sigma^d - \mu^d,
\end{equation}
where $\beta \geq 0$ is the tradeoff weight between mean and
variance. Therefore, our goal is to solve the \emph{mean-variance
optimization} problem:
\begin{equation}\label{eq_problem_origin}
\hspace{-5cm} \mbox{(P0): \hspace{4cm} }
\begin{array}{cc}
&\eta^* = \min\limits_{d \in \mathcal D}\{ \beta \sigma^d - \mu^d \}, \\
&d^* \in \argmin\limits_{d \in \mathcal D}\{ \beta \sigma^d -
\mu^d\}.
\end{array}
\end{equation}
Note that $\eta^*$ and $d^*$ depend on $\beta$, and we may also use
$\eta^*(\beta)$ and $d^*(\beta)$ if necessary. In
Fig.~\ref{fig_pareto}, the red star points are Pareto efficient
solutions which dominate the black dot solutions. The dashed curve
is the Pareto frontier which can be obtained by solving
\eqref{eq_problem_origin} with different $\beta \geq 0$. We can also
observe that the dashed line is tangent to the Pareto frontier,
where the slope is $\beta$ and the tangent point is
$(\sigma^{d^*(\beta)}, \mu^{d^*(\beta)})$.

\begin{figure}[htbp]
\centering
\includegraphics[width=.65\columnwidth]{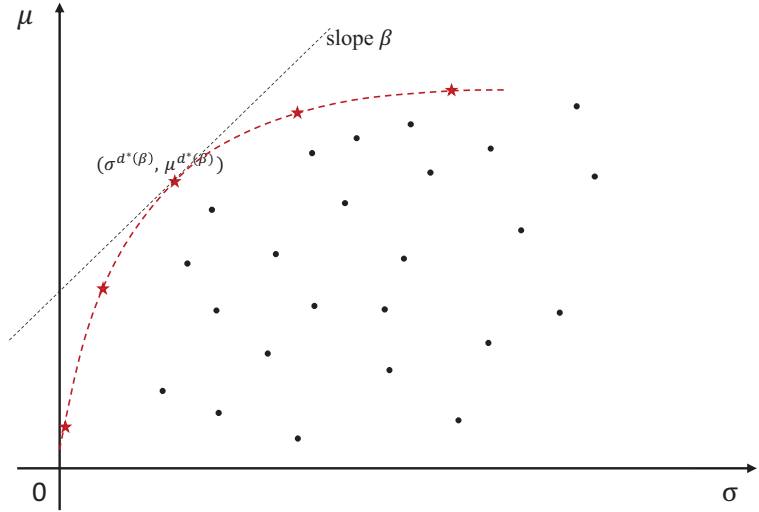}
\caption{Illustration of the Pareto frontier in mean-variance
optimization.}\label{fig_pareto}
\end{figure}

How to efficiently solve \eqref{eq_problem_origin} is the key of the
mean-variance optimization in MDPs. Since the variance function
$(r(i,d(i)) - \mu^d)^2$ depends on history and future behaviors
through $\mu^d$, it is not either additive or Markovian. The
mean-variance optimization problem \eqref{eq_problem_origin} does
not fit a standard model of MDPs and the principle of dynamic
programming fails \citep{Chung94,Sobel94,Xia16a}. Although there is
a recent progress on this problem by using the technique of
sensitivity-based optimization instead of the traditional dynamic
programming \citep{Xia20}, it can only find a local optimum of this
mean-variance optimization problem. A local optimum is not
guaranteed as a Pareto efficient solution. Thus, finding the global
optimum of \eqref{eq_problem_origin} is still an unsolved problem in
the mean-variance optimization of MDPs and we accomplish this
challenge in the rest of this paper.

\section{Main Results}\label{section_result}
%We utilize the sensitivity-based optimization theory (perturbation
%analysis of MDPs by \cite{Cao07}) to study this problem. Some
%results can be referred to our previous work \citep{Xia20} and we
%introduce the necessary parts briefly for self-containment.

First, we introduce the concept of \emph{pseudo mean} and
\emph{pseudo variance} of an MDP under policy $d \in \mathcal D$
\citep{Xia16a,Xia20}:
\begin{equation}\label{eq_psuvar}
\tilde{\sigma}^d (y) = \bm \pi^d (\bm r^d - y \bm e)^2_{\odot},
\qquad y \in \mathbb R,
\end{equation}
where $\tilde{\sigma}^d (y)$ is called the pseudo variance of the
MDP with the pseudo mean $y$. We can derive that the difference
between the pseudo variance and the real variance is
\begin{eqnarray}\label{eq_deltasigma}
\Delta^d(y) &:=& \tilde{\sigma}^d (y) - \sigma^d = \bm \pi^d (\bm
r^d -
y \bm e)^2_{\odot} - \bm \pi^d (\bm r^d - \mu^d \bm e)^2_{\odot} \nonumber\\
&=& \sum_{i \in \mathcal S} \pi^d(i) [(r(i,d(i))-y)^2 -
(r(i,d(i))-\mu^d)^2] \nonumber\\
&=& (y - \mu^d)^2 \geq 0.
\end{eqnarray}
We call $\Delta^d(y)$ the \emph{variance distortion} caused by the
pseudo mean $y$. Interestingly, we observe that the pseudo variance
$\tilde{\sigma}^d (y)$ is a \emph{convex quadratic} function of $y$,
since $\tilde{\sigma}^d (y) = \sigma^d + (y - \mu^d)^2$. When the
pseudo mean $y$ equals the real mean $\mu^d$, the variance
distortion is zero and the pseudo variance attains its minimum which
is exactly the real variance, i.e.,
\begin{equation}\label{eq_psuvarmin}
\sigma^d = \min_{y \in \mathbb R}\tilde{\sigma}^d (y) =
\tilde{\sigma}^d (y^*)\Big|_{y^*=\mu^d}.
\end{equation}

\noindent\textbf{Remark~1.} The above property of variance is
analogous to CVaR (Conditional Value at Risk) discovered by
\cite{Rockafellar00}: the CVaR of random variable $X$ at probability
level $\alpha$ is defined as $\CVaR_{\alpha}(X):= \mathbb E[X|X \geq
F^{-1}_{X}(\alpha)]$, and equals $\min\limits_{y \in \mathbb R}
\mathbb E[y + \frac{1}{1-\alpha}[X-y]^+]$, where $F^{-1}_{X}(\cdot)$
is the inverse distribution function of $X$,
$[X-y]^+:=\max\{0,X-y\}$, $\mathbb E[y + \frac{1}{1-\alpha}[X-y]^+]$
is a convex function of $y$ and its minimum attains at $y^* =
F^{-1}_{X}(\alpha)$.

With this property \eqref{eq_psuvarmin}, we can convert the original
mean-variance optimization problem to a bilevel MDP problem and
directly derive the following lemma.
\begin{lemma}[Bilevel MDP]\label{lemma1}
The mean-variance optimization problem \eqref{eq_problem_origin} is
equivalent to a bilevel MDP problem where the inner one is a
standard MDP with cost function $\beta (\bm r - y \bm e)^2_{\odot} -
\bm r$:
\begin{equation}\label{eq_bilevel}
\eta^* = \min_{d \in \mathcal D}\{ \beta \sigma^d - \mu^d \} =
\min_{y \in \mathbb R}\min_{d \in \mathcal D}\{\beta
\tilde{\sigma}^d (y) - \mu^d \}.
\end{equation}
\end{lemma}
The above bilevel MDP formulation is similar to the Fenchel duality
formulation \citep{Xie18}, while our formulation \eqref{eq_bilevel}
naturally comes from \eqref{eq_deltasigma} of pseudo variance which
was originally discovered by \cite{Xia16a}. The inner problem
$\min\limits_{d \in \mathcal D}\{\beta \tilde{\sigma}^d (y) - \mu^d
\}$ aims to optimize the pseudo mean-variance, which is a standard
MDP denoted by tuple $\mathcal M(y) := \langle \mathcal S, \mathcal
A, \mathcal P, \beta(\bm r - y \bm e)^2_{\odot} - \bm r \rangle$. We
can use traditional dynamic programming to solve this MDP. For
different outer variable of pseudo mean $y$, we have to solve
different MDP $\mathcal M(y)$. The number of solving inner MDPs is
equal to the number of $y \in \mathbb R$, which is computationally
intractable. Therefore, efficiently solving this bilevel MDP problem
\eqref{eq_bilevel} is challenging.

With \eqref{eq_psuvarmin}, we see that the optimal $y^*$ in
\eqref{eq_bilevel} satisfies $y^* = \mu^{d^*}$ for the optimal
policy $d^* \in \mathcal D$. Therefore, we can restrict $y$'s value
domain from $y \in \mathbb R$ to a much smaller set $y \in \mathcal
Y$, where $\mathcal Y:= \{\mu^d : \forall d \in \mathcal D \}$.
Although $\mathcal Y$ is still computationally intractable, we know
that $\mathcal Y \subset [\underline{r}, \overline{r}]$, where
$\underline{r}:= \min\limits_{i \in \mathcal S, a \in \mathcal A}
\{r(i,a)\}$ and $\overline{r}:= \max\limits_{i \in \mathcal S, a \in
\mathcal A}\{r(i,a)\}$. Therefore, the bilevel MDP
problem~\eqref{eq_bilevel} can be rewritten as
\begin{equation}\label{eq_bilevel2}
\eta^* = \min_{d \in \mathcal D}\{ \beta \sigma^d - \mu^d \} =
\min_{y \in [\underline{r}, \overline{r}]}\min_{d \in \mathcal
D}\{\beta \tilde{\sigma}^d (y) - \mu^d \}.
\end{equation}

Since the inner problem $\mathcal M(y^*)$ is a standard MDP, we can
derive a concise proof about the optimality of deterministic
policies (detailed proofs can also be referred to
\cite{Haskell13,Xia20}): Suppose $(y^*,d^*)$ is an optimal solution
to \eqref{eq_bilevel2}. It is well known that there exists a
deterministic policy $d_0$ which attains the minimum of standard MDP
$\min_{d}\{\beta \tilde{\sigma}^d (y^*) - \mu^d \}$. It is obvious
that $y^*$ must be the real mean of the MDP under policy $d_0$.
Thus, $\beta \tilde{\sigma}^{d_0} (y^*) - \mu^{d_0} = \beta
\sigma^{d_0} - \mu^{d_0}$, which indicates that the deterministic
policy $d_0$ attains the minimum of mean-variance performance.

When the pseudo mean $y$ is fixed, the inner standard MDP $\mathcal
M(y)$ is an auxiliary problem, and its long-run average performance
under policy $d$ is a combined performance $\tilde{\eta}^d (y) =
\beta \tilde{\sigma}^d(y) - \mu^d$. We also call $\mathcal M(y)$ a
\emph{pseudo mean-variance optimization} problem:
\begin{equation}\label{eq_auxiMDP}
\hspace{-5cm} (\mathcal M(y)): \hspace{3cm}
\begin{array}{cc}
&\tilde{\eta}^*(y) = \min\limits_{d \in \mathcal D}\{\beta
\tilde{\sigma}^d
(y) - \mu^d \}. \\
&\tilde{d}^*(y) \in \argmin\limits_{d \in \mathcal D}\{\beta
\tilde{\sigma}^d (y) - \mu^d \}.
\end{array}
\end{equation}
For notation simplicity, sometimes we may omit $y$, and use
$\tilde{\eta}^*$ and $\tilde{d}^*$ if no confusion caused.
Therefore, the bilevel MDP \eqref{eq_bilevel2} for mean-variance
optimization can be rewritten as below.
\begin{equation*}
\eta^* = \min_{y \in [\underline{r}, \overline{r}]} \{
\tilde{\eta}^*(y) \}.
\end{equation*}
If we plot a curve of $\tilde{\eta}^*(y)$ with respect to $y$, %as
%illustrated in Fig.~\ref{fig_etay},
we can observe that $\eta^*$ is the global minimum of this curve at
point $y^*$ and the corresponding $\tilde{d}^*(y^*)$ is the optimal
policy of the original problem \eqref{eq_problem_origin}. From the
sensitivity analysis of MDPs, we can derive the following lemma.

\begin{lemma}[Critical points]\label{lemma2SA}
There exists a series of intervals $[y^c_{k-1}, y^c_{k}]$ with
$\bigcup\limits_{k=1,\dots,K}[y^c_{k-1}, y^c_{k}]=[\underline{r},
\overline{r}]$, in which the optimal policy of $\mathcal M(y)$
remains unvaried as $\tilde{d}^*_k := \tilde{d}^*(y)$ when $y \in
[y^c_{k-1}, y^c_{k}]$.
\end{lemma}

\begin{proof}
We rewrite the standard MDP problem $\mathcal M(y)$ as a linear
programming (LP) model:
\begin{equation} \label{eq_LP}
\begin{array}{rcl}
\tilde{\eta}^*(y) &=& \min\limits_{\bm x} \left\{ \sum\limits_{i \in
S}\sum\limits_{a \in \mathcal A} [\beta(r(i,a) - y)^2 -
r(i,a)]x(i,a) \right\} \\
&\mbox{s.t., }& \sum\limits_{a \in \mathcal A}x(i,a) =
\sum\limits_{j \in S}\sum\limits_{a \in
\mathcal A} p(i|j,a) x(j,a), \qquad \forall i \in \mathcal S, \\
&&\sum\limits_{i \in S}\sum\limits_{a \in \mathcal A} x(i,a) = 1, \\
&&x(i,a) \geq 0, \qquad \forall i \in \mathcal S,  a \in \mathcal A.
\end{array}
\end{equation}
The above problem can be represented as a standard LP model:
\begin{equation}\label{eq_PLP}
\tilde{\eta}^*(y) - \beta y^2 = \min_{\bm x}\{(\bm c + y\bm c')^\T
\bm x | \bm A \bm x = \bm b, \bm x \geq \bm 0\},
\end{equation}
where we utilize the fact $\sum\limits_{i\in \mathcal S,a\in
\mathcal A} \beta y^2 x(i,a) = \beta y^2$, the $S$-by-$SA$ matrix
$\bm A$ and the $S$-dimensional column vector $\bm b$ are determined
by the constraint equations in \eqref{eq_LP}, $\bm c = \beta \bm
r^2_{\odot} - \bm r$, $\bm c' = -2\beta \bm r$, $\bm r$ and $\bm x$
are $SA$-dimensional column vector with element $r(i,a)$ and
$x(i,a)$, respectively. We observe that the right-hand-side of
\eqref{eq_PLP} is a \emph{parametric linear programming} (PLP)
\citep{Gal97,Tan11} with a linear parameter $y$. Below we do
sensitivity analysis for this PLP problem. For a given $y$, suppose
$\bm x^*_k$ is the optimal solution of \eqref{eq_PLP} and its
associated basis matrix is $\bm B$. We can verify that $\bm x_k^*$
in this LP is equivalent to the optimal policy $\tilde{d}^*_k$ of
the MDP $\mathcal M(y)$, where the optimal action in state $i$ is
$\tilde{d}^*_k(i) \in \argmax\limits_{a \in \mathcal
A}\{x^*_k(i,a)\}$, $i \in \mathcal S$. With the terminology of LP,
we denote $\bm A = [\bm B, \bm N]$, $\bm x = [\bm x_B; \bm x_N]$,
$\bm c = [\bm c_B; \bm c_N]$, and $\bm c' = [\bm c'_B; \bm c'_N]$.
The optimality test of the simplex method requires that all the test
coefficients should be nonpositive, i.e.,
\begin{equation*}
(\bm c_B + y \bm c'_B)^\T \bm B^{-1} \bm A - (\bm c + y \bm c')^\T =
(\bm c_B^\T \bm B^{-1} \bm A - \bm c^\T) + y ({\bm c'_B}^\T \bm
B^{-1} \bm A - {\bm c'}^\T ) \leq \bm 0.
\end{equation*}
For notation simplicity, we denote the $SA$-dimensional test
coefficients vector as
\begin{eqnarray}
\bm \zeta^\T &:=& \bm c_B^\T \bm B^{-1} \bm A - \bm c^\T, \label{eq_LPzeta1}\\
{\bm \zeta'}^\T &:=& {\bm c'_B}^\T \bm B^{-1} \bm A - {\bm c'}^\T.
\label{eq_LPzeta2}
\end{eqnarray}
In order to find the interval $[y^c_{k-1}, y^c_{k}]$ that any $y$
therein makes $\bm x^*_k$ remain optimal, we only need to solve $y$
satisfying $\bm \zeta + y \bm \zeta' \leq \bm 0$. It is easy to
verify that the solution is
\begin{eqnarray}
y^c_{k-1} = \max_{i\in \mathcal S, a\in \mathcal A}\left\{
-\frac{\zeta(i,a)}{\zeta'(i,a)}\Bigg| \zeta'(i,a)
< 0 \right\}  \qquad (\max\{\varnothing\} = -\infty), \label{eq_yk-1} \\
y^c_{k} = \min_{i\in \mathcal S, a\in \mathcal A}\left\{
-\frac{\zeta(i,a)}{\zeta'(i,a)}\Bigg| \zeta'(i,a) > 0 \right\}
\qquad (\min\{\varnothing\} = +\infty). \ \label{eq_yk}
\end{eqnarray}
Obviously, we can first let $y=\underline{r}$, and use
\eqref{eq_yk-1} and \eqref{eq_yk} to obtain $y^c_0 = -\infty$ and
$y^c_1$, respectively. Other $y^c_k$'s can be computed sequentially,
and $y^c_K = + \infty$. The lemma is proved.
\end{proof}

We call such $y^c_k$'s in Lemma~\ref{lemma2SA} \emph{critical
points} for the sensitivity analysis of MDP $\mathcal M(y)$, and
$K+1$ is the number of critical points. Actually, by using the
specific structures of $\bm A, \bm b, \bm r$ of the LP for $\mathcal
M(y)$, we can verify that
\begin{equation}
[\bm B^{-1}]^\T = (\bm I - \bm P^d + \bm e \bm e^\T)^{-1},
\end{equation}
which is a generalized \emph{fundamental matrix} in MDPs
\citep{Xia16}, where the policy $d$ corresponds to the vector of
feasible basic variables $\bm x_B$ of the basis matrix $\bm B$. The
associated vector $\bm b$ equals $\bm e$. The matrix $\bm A$ has a
similar structure of $\bm B$, i.e., $\bm A = \bm I_e - \bm P^\T_{e}
+ \bm e \bm e^\T$, where $\bm I_e$ is an $S$-by-$SA$ matrix whose
element $I_e(i,(j,a))=1 / 0$ if $i=j$/otherwise, $\bm P_e$ is an
$SA$-by-$S$ matrix whose element $P_e((i,a),j) = p(j|i,a)$, $\bm e
\bm e^\T$ is an $S$-by-$SA$ matrix of 1's. We also observe that $\bm
c_B$ is the column vector of the cost function of the MDP under
policy $d$ (associated with $\bm x_B$). We can derive that $\bm
c^\T_B \bm B^{-1}$ is equal to the \emph{performance potential} or
\emph{relative value function} in MDPs \citep{Cao07,Puterman94}.
Thus, we can verify that $\bm c_B^\T \bm B^{-1}$ in
\eqref{eq_LPzeta1} coincides with the \emph{Poisson equation} in
MDPs
\begin{eqnarray*}
\bm g^d &=& [\bm B^{-1}]^\T \bm c_B = (\bm I - \bm P^d + \bm e \bm
e^\T)^{-1} (\beta {\bm r^d}^2_{\odot} - \bm r^d), \\
{\bm g'}^d &=& [\bm B^{-1}]^\T \bm c'_B = -2 \beta (\bm I - \bm P^d
+ \bm e \bm e^\T)^{-1} \bm r^d,
\end{eqnarray*}
where $\bm g^d$ and ${\bm g'}^d$ are $S$-dimensional column vector
of performance potentials for the MDP under policy $d$ with cost
function $\beta {\bm r^d}^2_{\odot} - \bm r^d$ and $-2 \beta \bm
r^d$, respectively. Thus, we can rewrite the element of
\eqref{eq_LPzeta1} and \eqref{eq_LPzeta2} as
\begin{eqnarray}
\zeta(i,a)  &:=& g^d(i) - \sum_{j \in \mathcal S} p(j|i,a) g^d(j) + \sum_{j \in \mathcal S} g^d(j) + r(i,a) - \beta {r(i,a)}^2 =: -\tilde{A}^d(i,a) , \label{eq_MDPzeta1}\\
\zeta'(i,a) &:=& {g'}^d(i) - \sum_{j \in \mathcal S} p(j|i,a)
{g'}^d(j) + \sum_{j \in \mathcal S} {g'}^d(j) + 2\beta r(i,a) =:
-\tilde{A'}^d(i,a) , \label{eq_MDPzeta2}
\end{eqnarray}
where we utilize a fact that $\bm e^\T \bm g^d$ equals the long-run
average performance of the MDP under policy $d$, which can be
verified from the Poisson equation. Therefore, we can substitute
\eqref{eq_MDPzeta1} and \eqref{eq_MDPzeta2} into \eqref{eq_yk-1} and
\eqref{eq_yk} to compute all the critical points $y^c_k$'s.

\noindent\textbf{Remark~2.} Equations~\eqref{eq_MDPzeta1} and
\eqref{eq_MDPzeta2} interestingly indicate that the test coefficient
$\zeta(i,a)$ in LP coincides with the \emph{advantage function}
$\tilde{A}^d(i,a)$ which is a key quantity widely used in
reinforcement learning \citep{Sutton18}. $\tilde{A}^d(i,a) < 0$
means that action $a$ at state $i$ induces a smaller long-run
average cost than the current policy $d$ in MDPs, which hints that
$\zeta(i,a) > 0$ and the corresponding variable $x(i,a)$ should be
an entering basic variable in LP.

By using \eqref{eq_deltasigma}, we can obtain the relation between
the pseudo and real mean-variance combined performances as
\begin{equation*} \label{eq_y22}
\tilde{\eta}^{d}(y) = \beta[\sigma^{d} + (y-\mu^{d})^2] - \mu^{d} =
\eta^{d} +  \beta (y-\mu^{d})^2, \quad \forall d \in \mathcal D, y
\in \mathbb R.
\end{equation*}
Since the optimal policy of $\mathcal M(y)$ remains the same as
$\tilde{d}^*_k$ for any $y \in [y^c_{k-1}, y^c_{k}]$, we have
\begin{equation} \label{eq_y23}
\begin{array}{rcll}
\tilde{\eta}^*(y) &=& \eta^{\tilde{d}^*_k} +  \beta
(y-\mu^{\tilde{d}^*_k})^2, \qquad &\forall y \in [y^c_{k-1},
y^c_{k}], \\
\tilde{\eta}^*(y) &=&
\min\limits_{k=1,\dots,K}\left\{\eta^{\tilde{d}^*_k} + \beta
(y-\mu^{\tilde{d}^*_k})^2\right\}, \qquad &\forall y \in \mathbb R.
\end{array}
\end{equation}
That is, each piece of curves in Fig.~\ref{fig_etay} is a convex
quadratic function of $y$, and the whole curve is the minimum of all
these quadratic functions. With \eqref{eq_y23}, it is interesting to
note that all the piecewise curves have the same shape (the same
term of $\beta y^2$) but different positions. At each critical point
$y^c_k$, we can validate that both $\tilde{d}^*_k$ and
$\tilde{d}^*_{k+1}$ are optimal policies of MDP $\mathcal M(y^c_k)$,
so $\tilde{\eta}^*(y)$ is continuous in $y$. Therefore, we directly
derive the following lemma.
\begin{lemma}\label{lemma2}
The pseudo mean-variance performance $\tilde{\eta}^*(y)$ is a convex
piecewise quadratic function and continuous in $y$, and its global
minimum is the optimal solution of \eqref{eq_problem_origin}.
\end{lemma}

\begin{figure}[htbp]
\centering
\includegraphics[width=.65\columnwidth]{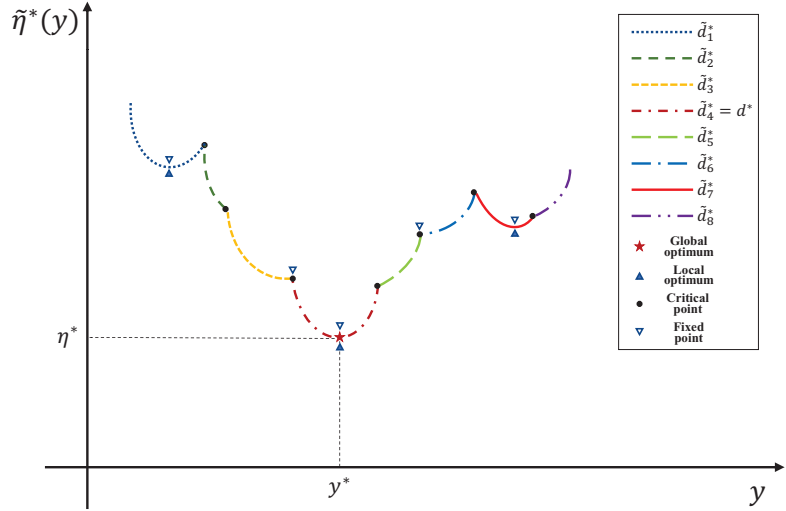}
\caption{Illustration of the convex piecewise quadratic curve of the
pseudo mean-variance performance $\tilde{\eta}^*(y)$ with respect to
the pseudo mean $y$.}\label{fig_etay}
\end{figure}

From Fig.~\ref{fig_etay}, we can observe that
$\min\limits_{y}\{\tilde{\eta}^*(y)\}$ is difficult to solve,
because $\tilde{\eta}^*(y)$ may have multiple local optima
$\hat{y}^*$ which has a zero derivative, i.e.,
\begin{equation*}
\frac{\partial \tilde{\eta}^*(y)}{\partial y}\Big|_{\hat{y}^*} = 2
\beta (\hat{y}^* - \mu^{\tilde{d}^*_k}) = 0, \qquad \mbox{if }
\hat{y}^* \in [y^c_{k-1}, y^c_k].
\end{equation*}
This indicates that a local optimum $\hat{y}^*$ must satisfy the
following \emph{fixed point equation}
\begin{equation}\label{eq_fixedpoint}
y = \mu^{\tilde{d}^*(y)}.
\end{equation}
Note that the fixed point solutions of \eqref{eq_fixedpoint}, as
indicated by ``inverted" triangles in Fig.~\ref{fig_etay}, are not
necessarily local optima of $\tilde{\eta}^*(y)$. The reason is when
a fixed point is also a critical point $y^c_k$, we can verify that
its left-derivative is 0 (or positive) and its right-derivative is
negative (or 0), and such point is not a local optimum $\hat{y}^*$
of $\tilde{\eta}^*(y)$, as illustrated by Fig.~\ref{fig_etay}. We
can verify that the policies indicated by all these fixed point
solutions of~\eqref{eq_fixedpoint} are exactly the so-called local
optimal policies in mixed or randomized policy space of MDPs, as
discovered by \cite{Xia20}. These two kinds of local optima are
different: local optima $\hat{y}^*$ in this paper are included by
local optimal policies (fixed point solutions) defined in
\cite{Xia20}, as illustrated by triangles and ``inverted" triangles
in Fig.~\ref{fig_etay}, respectively.

Note that $\tilde{d}^*$ is optimal only for the pseudo
problem~\eqref{eq_auxiMDP}, not for the original
problem~\eqref{eq_problem_origin}. Fortunately, we discover that
$\tilde{d}^*$ has a better mean-variance performance than some other
policies, which is described by the following lemma.

\begin{lemma}[Policy dominance]\label{lemma3}
For any $y \in \mathbb R$, $\tilde{d}^*$ is an optimal policy of the
MDP $\mathcal M(y)$ in \eqref{eq_auxiMDP}. If a policy $d \in
\mathcal D$ satisfying $\mu^d \in [y-|y-\mu^{\tilde{d}^*}|, \
y+|y-\mu^{\tilde{d}^*}|]$, then
\begin{equation}\label{eq_worseq}
\beta \sigma^{\tilde{d}^*} - \mu^{\tilde{d}^*} \leq \beta \sigma^d -
\mu^d.
\end{equation}
\end{lemma}

\begin{proof}
Since $\tilde{d}^*$ is an optimal policy of the standard MDP problem
$\mathcal M(y)$, we have
\begin{equation}\label{eq9}
\beta \tilde{\sigma}^{\tilde{d}^*} (y) - \mu^{\tilde{d}^*}  \leq
\beta \tilde{\sigma}^d (y) - \mu^d, \quad \forall d \in \mathcal D.
\end{equation}
With \eqref{eq_deltasigma}, we derive
\begin{equation}
\tilde{\sigma}^d (y) = \sigma^d + (y - \mu^d)^2, \quad \forall d \in
\mathcal D. \nonumber
\end{equation}
Substituting the above equation into \eqref{eq9}, we have
\begin{equation}\label{eq10}
\beta \sigma^{\tilde{d}^*}  - \mu^{\tilde{d}^*} + \beta (y -
\mu^{\tilde{d}^*})^2 \leq \beta \sigma^d  - \mu^d + \beta (y -
\mu^d)^2, \quad \forall d \in \mathcal D.
\end{equation}
For any policy $d$ satisfying $\mu^d \in [y-|y-\mu^{\tilde{d}^*}|, \
y+|y-\mu^{\tilde{d}^*}|]$, we have
\begin{equation}\label{eq12ineq}
(y - \mu^{\tilde{d}^*})^2 \geq (y - \mu^d)^2.
\end{equation}
Substituting \eqref{eq12ineq} into \eqref{eq10}, we directly obtain
\eqref{eq_worseq}, and the lemma is proved.
\end{proof}

Moreover, if $d$ satisfies $\mu^d \in (y-|y-\mu^{\tilde{d}^*}|, \
y+|y-\mu^{\tilde{d}^*}|)$, we have $(y - \mu^{\tilde{d}^*})^2 > (y -
\mu^d)^2$, and the inequality in \eqref{eq_worseq} strictly holds.
Therefore, Lemma~\ref{lemma3} indicates that $\tilde{d}^*$ dominates
all the policies whose means lie in the interval
$[y-|y-\mu^{\tilde{d}^*}|, \ y+|y-\mu^{\tilde{d}^*}|]$, and these
dominated policies can be removed from the policy space $\mathcal D$
to save computation. We illustrate this property by
Fig.~\ref{fig_dominatedarea}, where we can see that the shadow area
can be discarded since the policies therein are always dominated by
$\tilde{d}^*$. Thus, we can utilize Lemma~\ref{lemma3} to
significantly reduce the complexity of the mean-variance problem
\eqref{eq_problem_origin}.

\begin{figure}[htbp]
\centering
\includegraphics[width=.65\columnwidth]{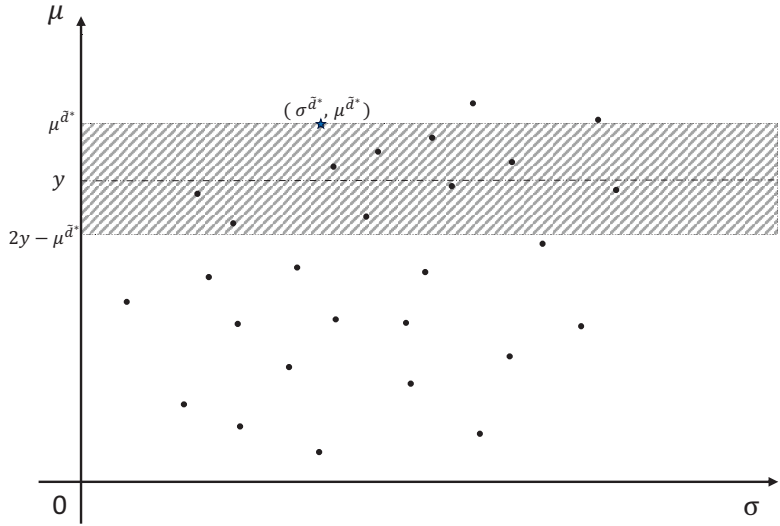}
\caption{Illustration of the dominated policy area indicated by
Lemma~\ref{lemma3}.}\label{fig_dominatedarea}
\end{figure}

With Lemma~\ref{lemma3}, we can develop an algorithm to iteratively
solve the bilevel MDP problem~\eqref{eq_bilevel}, which is described
by Algorithm~\ref{algo1}. The key idea is to solve a series of
auxiliary problems $\mathcal M(y)$'s, and repeatedly reduce the
auxiliary variable $y$'s value domain $\mathbb Y$ by using
Lemma~\ref{lemma3}. When $\mathbb Y$ is shrunk to be empty, the
best-so-far solution of $\mathcal M(y)$'s is the global optimum of
the bilevel MDP problem \eqref{eq_bilevel} or \eqref{eq_bilevel2}.
The global convergence of Algorithm~\ref{algo1} is guaranteed by the
following theorem.

\begin{algorithm}[htbp]
 \caption{An iterative algorithm to find globally optimal policies of mean-variance MDPs.}\label{algo1}
\begin{algorithmic}[1]

\State \textbf{Input:} MDP parameters $\langle \mathcal S, \mathcal
A, \mathcal P, \bm r \rangle$ and coefficient $\beta$.

\State \textbf{Output:} The globally optimal policy $d^*$ and its
optimum $\eta^*$.

\State \hspace{0cm} \emph{initialization}: set $\mathbb Y =
[\underline{r}, \overline{r}]$, $l = 0$, and $\eta^* = +\infty$.

\While {$\mathbb Y \neq \varnothing$}

\State \hspace{-0.6cm} \emph{set pseudo mean}: $\mathbb Y$ is
composed of multiple intervals, descendingly sort these intervals as
$\mathbb Y = \{\mathbb Y_1, \mathbb Y_2, \dots\}$, set $y_l$ as the
median value of the first interval of $\mathbb Y$, i.e., $y_l =
(\max\{\mathbb Y_1\} + \min\{\mathbb Y_1\})/2$.

\State \hspace{-0.6cm} \emph{solve auxiliary problem}: solve the
standard MDP $\mathcal M(y_l)$ by using traditional dynamic
programming methods, such as policy iteration or value iteration,
and obtain an optimal policy $\tilde{d}^*(y_l)$ of $\mathcal
M(y_l)$.

\State \hspace{-0.6cm} \emph{domain shrinking}: use
Lemma~\ref{lemma3} to remove the interval
$[y_l-|y_l-\mu^{\tilde{d}^*(y_l)}|, \
y_l+|y_l-\mu^{\tilde{d}^*(y_l)}|]$ from the value domain $\mathbb
Y$, i.e., $\mathbb Y = \mathbb Y -
[y_l-|y_l-\mu^{\tilde{d}^*(y_l)}|, \
y_l+|y_l-\mu^{\tilde{d}^*(y_l)}|]$.

\State \hspace{-0.6cm} \emph{update parameters}:

\If{$\eta^* > \beta \sigma^{\tilde{d}^*(y_l)} -
\mu^{\tilde{d}^*(y_l)}$}

\State $\eta^* = \beta \sigma^{\tilde{d}^*(y_l)} -
\mu^{\tilde{d}^*(y_l)}$.

\State $d^* = \tilde{d}^*(y_l)$.

\EndIf

\State set $l := l+1$.

\EndWhile

\Return $d^*$ and $\eta^*$.

\end{algorithmic}
\end{algorithm}

\begin{theorem}\label{theorem1}
Algorithm~\ref{algo1} converges to the global optimum of the
mean-variance optimization problem after a finite number of
iterations.
\end{theorem}
\begin{proof}
To prove the convergence of Algorithm~\ref{algo1}, we only need to
prove that the value domain $\mathbb Y$ is reduced to an empty set
after a finite number of iterations. From the algorithm procedure,
we can observe that for each iteration of solving an auxiliary
problem $\mathcal M(y_l)$, we will derive a policy
$\tilde{d}^*(y_l)$ and remove a square area with y-axis interval
$[y_l - |y_l-\mu^{\tilde{d}^*(y_l)}|, \ y_l + |y_l -
\mu^{\tilde{d}^*(y_l)}|]$, as stated by Lemma~\ref{lemma3}. This
removed area at least contains the policy $\tilde{d}^*(y_l)$, as
illustrated by the 1st and 2nd iterations in
Fig.~\ref{fig_shrinkingproc}. Usually, it contains multiple policies
dominated by the policy $\tilde{d}^*(y_l)$, as illustrated by
Fig.~\ref{fig_dominatedarea}. If the current policy
$\tilde{d}^*(y_l)$ has already been removed by previous domain
shrinking operations, the current domain shrinking operation will
remove at least the interval $\mathbb Y_1$, as illustrated by the
3rd, 4th, and 5th iterations in Fig.~\ref{fig_shrinkingproc}, which
can be verified by the fact of $y_l$ being the median value of
$\mathbb Y_1$ and Lemma~\ref{lemma3}. In summary, each domain
shrinking operation will either delete at least one policy (not
deleted previously) or delete at least one interval $\mathbb Y_1$.
It is easy to verify that the largest number of intervals is
$|\mathcal D|+1$, where each iteration only deletes a very small
area around $\tilde{d}^*(y_l)$. Therefore, in the worst case, we
need $|\mathcal D|$ iterations to delete every policy and $|\mathcal
D|+1$ iterations to delete every interval. The algorithm will stop
after at most $2|\mathcal D|+1$ iterations.

Since each $\tilde{d}^*(y_l)$ dominates all the other policies
located in the shrunk area of the $l$-th iteration, the best-so-far
solution among all $\tilde{d}^*(y_l)$'s is the global optimum of the
original mean-variance optimization problem. This completes the
proof.
\end{proof}

\begin{figure}[htbp]
\centering
\includegraphics[width=1.0\columnwidth]{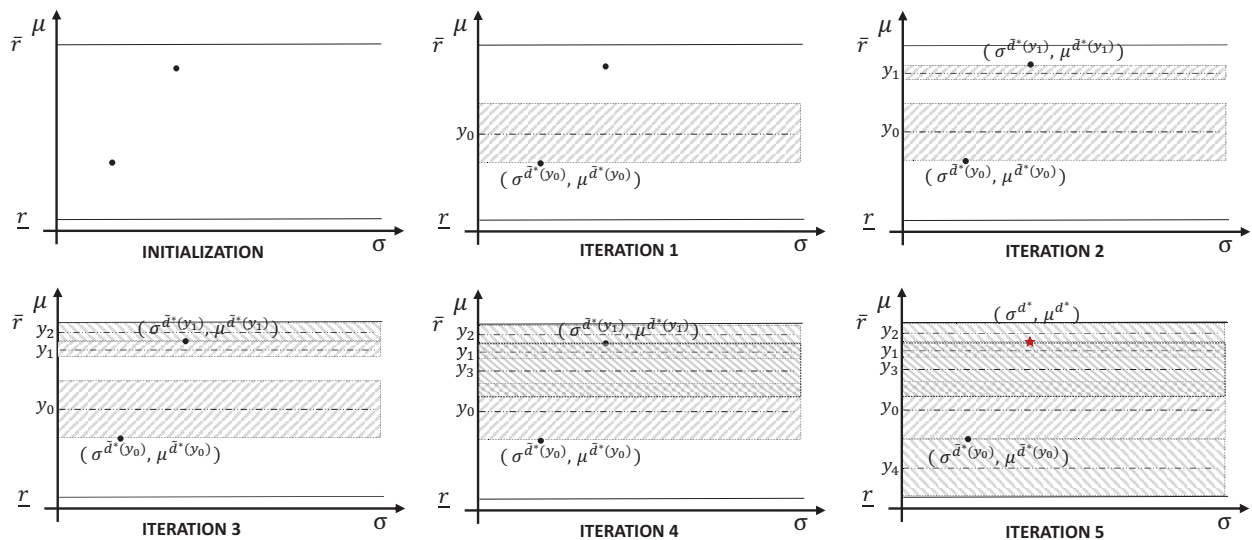}
\caption{A worst-case illustration of domain shrinking procedure of
Algorithm~\ref{algo1} for an example of a policy space with only two
solutions.}\label{fig_shrinkingproc}
\end{figure}

Fig.~\ref{fig_shrinkingproc} gives an illustration of the worst case
for an example of a policy space with only 2 solutions, it requires
$2\times2+1=5$ iterations to cover the whole interval
$[\underline{r}, \overline{r}]$. From the proof of
Theorem~\ref{theorem1}, we directly derive the following corollary
about the computational complexity of Algorithm~\ref{algo1}.

\begin{corollary}
The computational complexity of Algorithm~\ref{algo1} is $2|\mathcal
D|+1$ times of solving $\mathcal M(y)$.
\end{corollary}

Although the above computational complexity is not attractive, it
accounts for the worst case. Numerical experiments in
Section~\ref{section_experiment} demonstrate that the convergence
speed of Algorithm~\ref{algo1} is very fast in most cases.

\noindent\textbf{Remark~3.} By changing the update rule of $y_l$ in
Algorithm~\ref{algo1}, we can obtain different version of
algorithms. One example is to let $y_{l+1} =
\mu^{\tilde{d}^*(y_{l})}$, i.e., the pseudo mean $y_{l+1}$ is set as
the real mean of the optimal policy $\tilde{d}^*(y_{l})$ of
$\mathcal M(y_l)$. Such revised algorithm is very similar to the
policy iteration algorithm for solving local optimality equation in
\cite{Xia20}, both converge to a fixed point solution to
\eqref{eq_fixedpoint}.

Besides Lemma~\ref{lemma3}, we may further improve the shrinking
efficiency of dominated areas by using other properties. From the
viewpoint of bi-objective optimization in Fig.~\ref{fig_pareto}, we
observe that the minimization of objective $\beta \sigma^{d} -
\mu^{d}$ is interpreted to find the last solution $d^*$ tangent with
the line of slope $\beta$ when the line is moving toward top-left.
All the solutions located at the down-right side of this line have a
worse objective $\beta \sigma^{d} - \mu^{d}$ than that of the
solution $d^*$. This fact is illustrated by
Fig.~\ref{fig_dominatedarea2}.

\begin{figure}[htbp]
\centering
\includegraphics[width=.6\columnwidth]{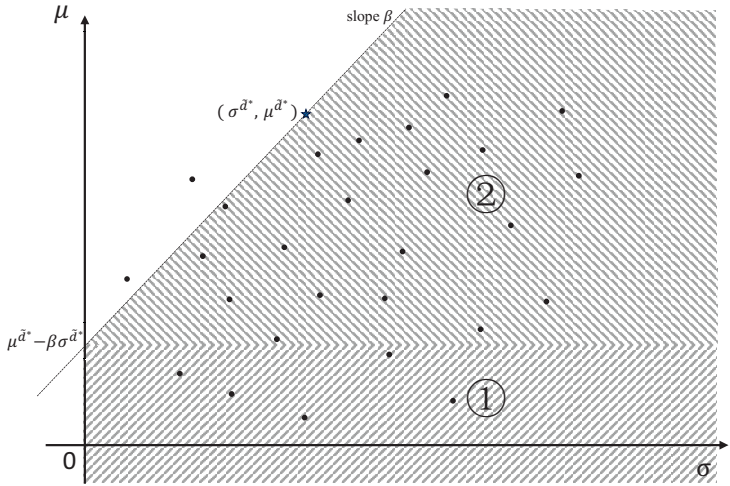}
\caption{Illustration of the dominated policy area indicated by
Lemma~\ref{lemma4}.}\label{fig_dominatedarea2}
\end{figure}

Therefore, based on an optimal policy $\tilde{d}^*$ by solving the
pseudo mean-variance MDP $\mathcal M(y)$, we directly derive the
following lemma about the shrinkage of dominated areas.
\begin{lemma}\label{lemma4}
For any policy $\tilde{d}^* \in \mathcal D$, the policies in the
following areas are dominated by $\tilde{d}^*$ and can be
discarded: \\
\textcircled{1} any policy $d \in \mathcal D$ satisfying $\mu^d \in (-\infty, \ \mu^{\tilde{d}^*} - \beta \sigma^{\tilde{d}^*}]$;  \\
\textcircled{2} any policy $d \in \mathcal D$ satisfying $\mu^d \in
[\mu^{\tilde{d}^*} - \beta \sigma^{\tilde{d}^*}, \ +\infty)$ and
$\sigma^d \in [\sigma^{\tilde{d}^*}+\frac{1}{\beta}(\mu^d -
\mu^{\tilde{d}^*}), \ +\infty)$.
\end{lemma}

The area~\textcircled{1} in Lemma~\ref{lemma4} is similar to the
area discarded by Lemma~\ref{lemma3}, both are square areas and have
no constraints on variances. Therefore, we can utilize the
rule~\textcircled{1} in Lemma~\ref{lemma4} to speed up the domain
shrinking of $\mathbb Y$ in Algorithm~\ref{algo1}. That is, at
line~7 of Algorithm~\ref{algo1}, we can add an extra operation to
discard the area~\textcircled{1} indicated by Lemma~\ref{lemma4}:
\begin{equation}\label{eq_Ydeleting}
\mathbb Y = \mathbb Y - (-\infty, \ \mu^{\tilde{d}^*(y_l)} - \beta
\sigma^{\tilde{d}^*(y_l)}].
\end{equation}
We call such algorithm revision \emph{Algorithm~\ref{algo1}-Plus},
whose performance is compared in our numerical experiments in
Section~\ref{section_experiment}. For example, in
Fig.~\ref{fig_shrinkingproc}, when $\beta$ is relatively small, the
Iteration~5 will be saved if we apply \eqref{eq_Ydeleting} for
policy $\tilde{d}^*(y_1)$ at the Iteration~2. This demonstrates that
Algorithm~\ref{algo1}-Plus is computationally saving compared with
Algorithm~\ref{algo1}.

\noindent\textbf{Remark~4.} It is easy to verify that all the
results in this paper can be extended to solely minimizing the
steady-state variance of MDPs. One trivial method is to let the
coefficient $\beta$ in \eqref{eq_problem_origin} large enough to
approximate the variance minimization of MDPs. Actually, if we
replace the mean-variance objective $\beta \sigma^d - \mu^d$ in
\eqref{eq_problem_origin} with the variance $\sigma^d$, we can
rigorously prove that all the previous results hold for this
variance minimization problem. Algorithm~\ref{algo1} also works to
find the optimal policy that attains the global minimum of the
variance in MDPs.

\section{Numerical Experiments}\label{section_experiment}
In this section, we validate the proposed algorithms with a
multi-period inventory control problem, where we consider both the
steady-state mean and variance of rewards. This problem is modeled
as an infinite-horizon discrete-time undiscounted MDP. The inventory
capacity is $C \in \mathbb{N}^+$. At each epoch $ t = 0, 1, \dots $,
the inventory level $ s_t \in \mathcal S = \{ 0, 1, \dots, C \} $ is
reviewed and an order $ a_t \in \mathcal A(s_t) = \{ 0, \dots, C -
s_t \} $ is made. The demands $\xi_t \sim B(C, p)$ are independent
and identically distributed, where $B(C, p)$ is a binomial
distribution and $ p $ is the probability of success. There is no
lead time and the next inventory level is determined as $s_{t+1} =
[s_t + a_t - \xi_t]^+$. The reward function is $r(s_t, a_t) =
\mathbb E[r(s_{t+1}|s_t,a_t)] = -\mathbb E\{b a_t + h s_{t+1} + l
[\xi_t - s_t - a_t]^+ \}$, where $ b,h $ and $ l $ are ordering,
holding and shortage costs per unit, respectively. By default, we
set $C = 4, p = 0.6, b = 1,  h = 0.7, l = 2.9,$ and $\beta = 10$. We
run algorithms 50 replications for statistical analysis.

Fig.~\ref{fig_convProc} illustrates an example of the convergence
process of Algorithm~\ref{algo1}, where the interval
$[\underline{r}, \overline{r}]$ is covered iteratively and the
global optimum is found after only 6 iterations. This demonstrates
the efficiency of Algorithm~\ref{algo1}, although the policy space
is large as $|\mathcal D| = (C+1)!$. %|\mathcal S|^{|\mathcal A|}

\begin{figure}[htbp]
    \centering
    \begin{minipage}{0.49\linewidth}
        \centering
        \includegraphics[width=1\linewidth]{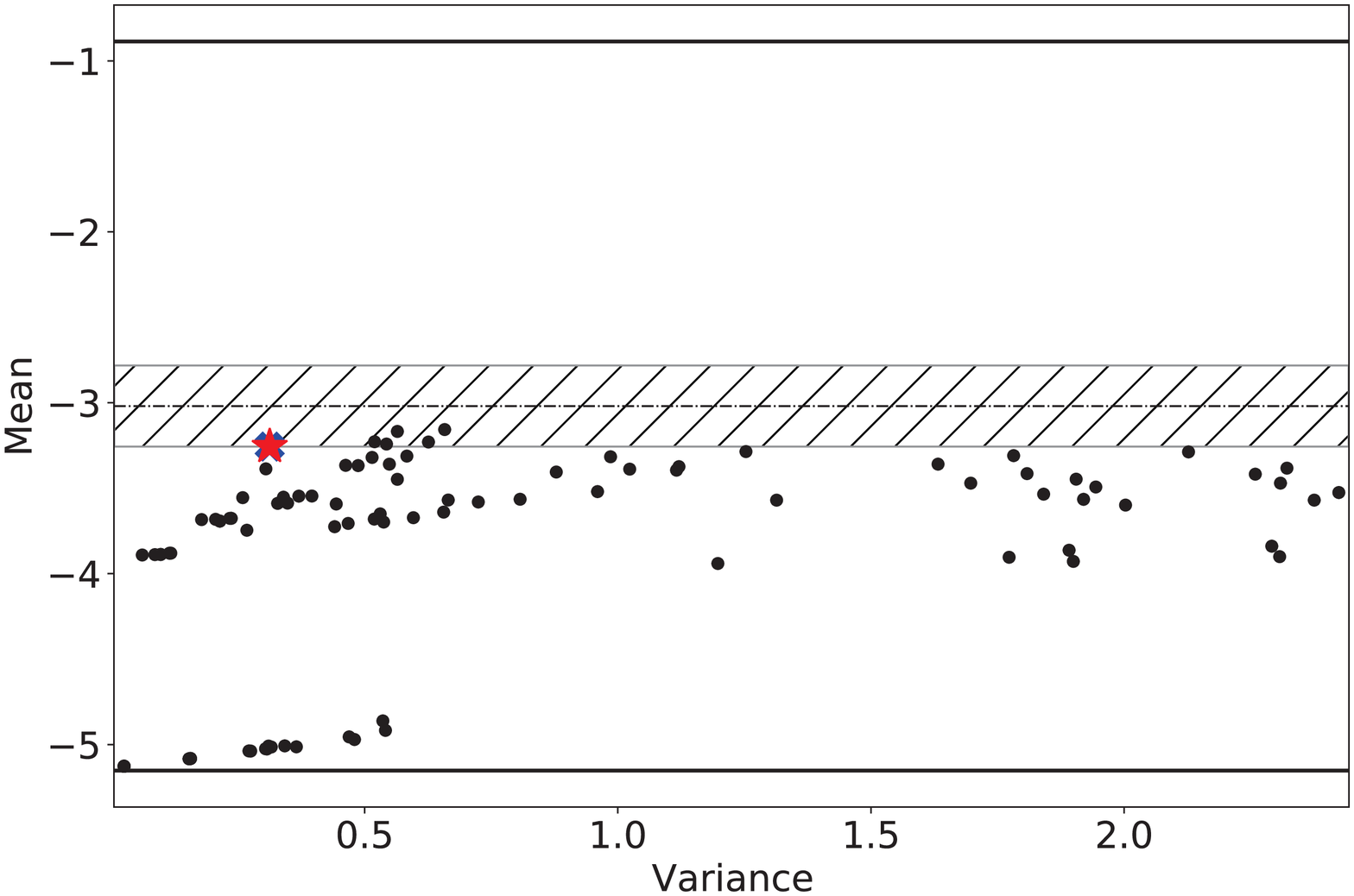}
        %       \caption{Step 1}
    \end{minipage}
    \begin{minipage}{0.49\linewidth}
        \centering
        \includegraphics[width=1\linewidth]{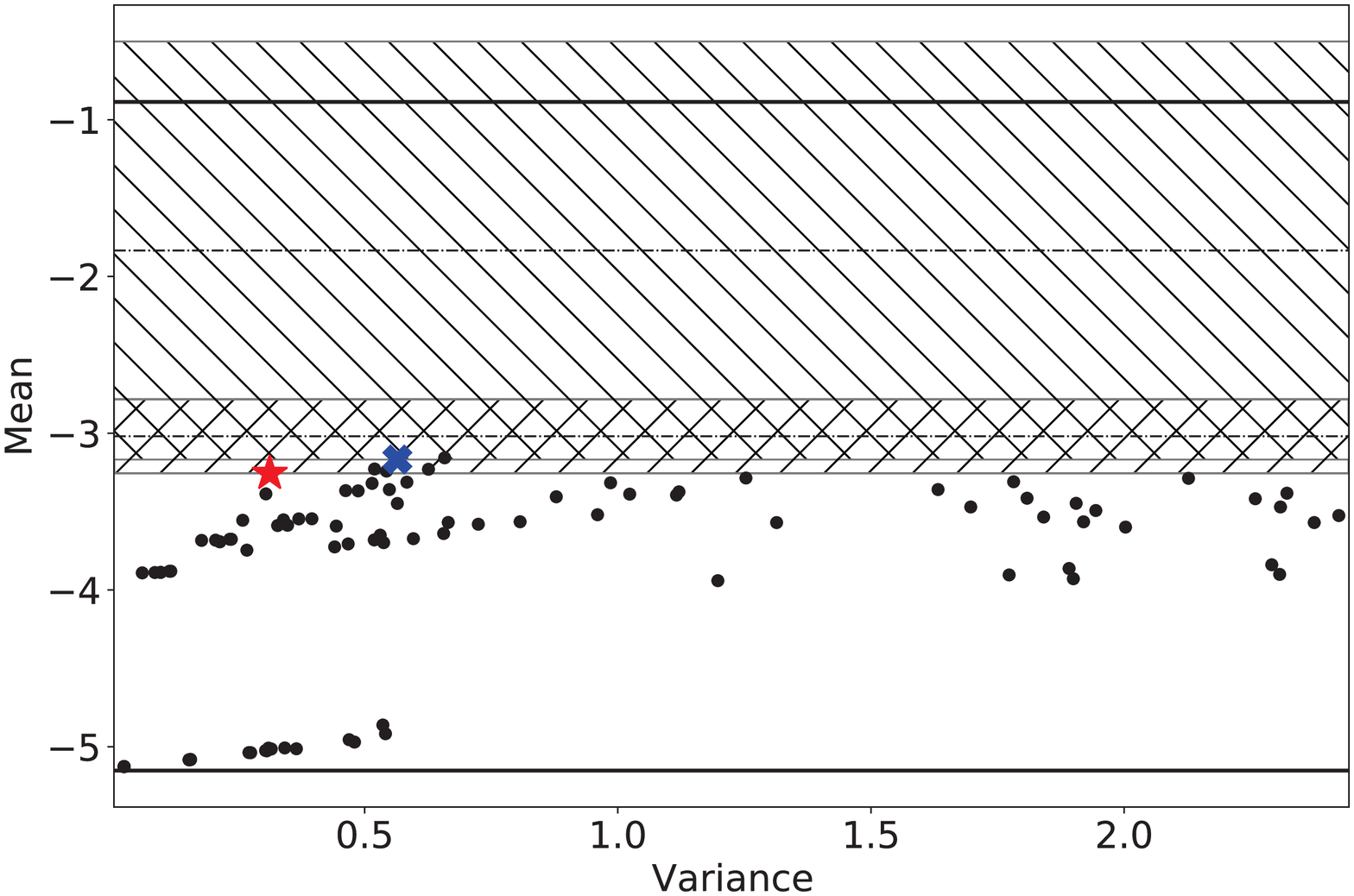}
        %       \caption{Step 2}
    \end{minipage}

    \begin{minipage}{0.49\linewidth}
        \centering
        \includegraphics[width=1\linewidth]{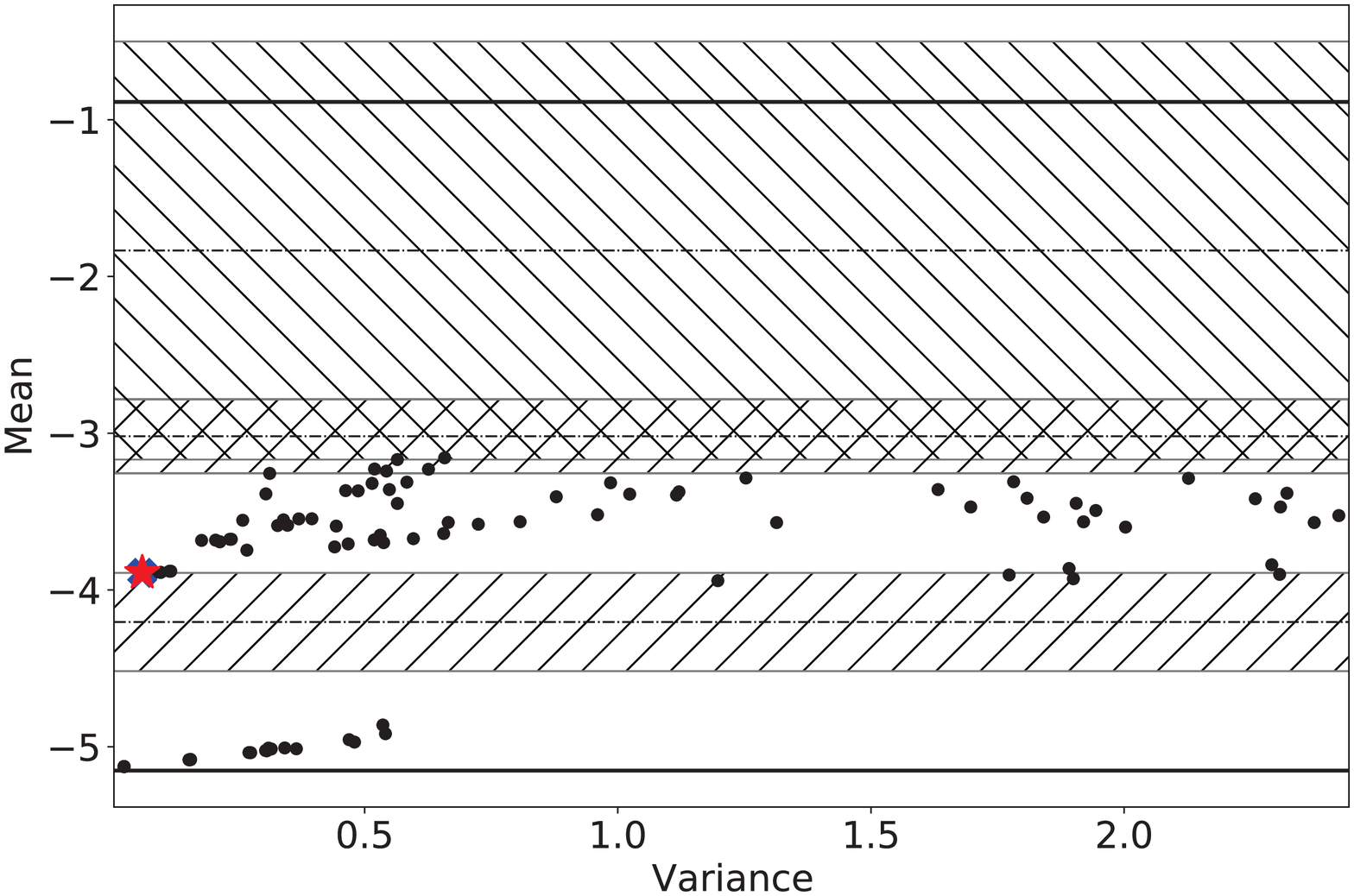}
        %       \caption{Step 3}
    \end{minipage}
    \begin{minipage}{0.49\linewidth}
        \centering
        \includegraphics[width=1\linewidth]{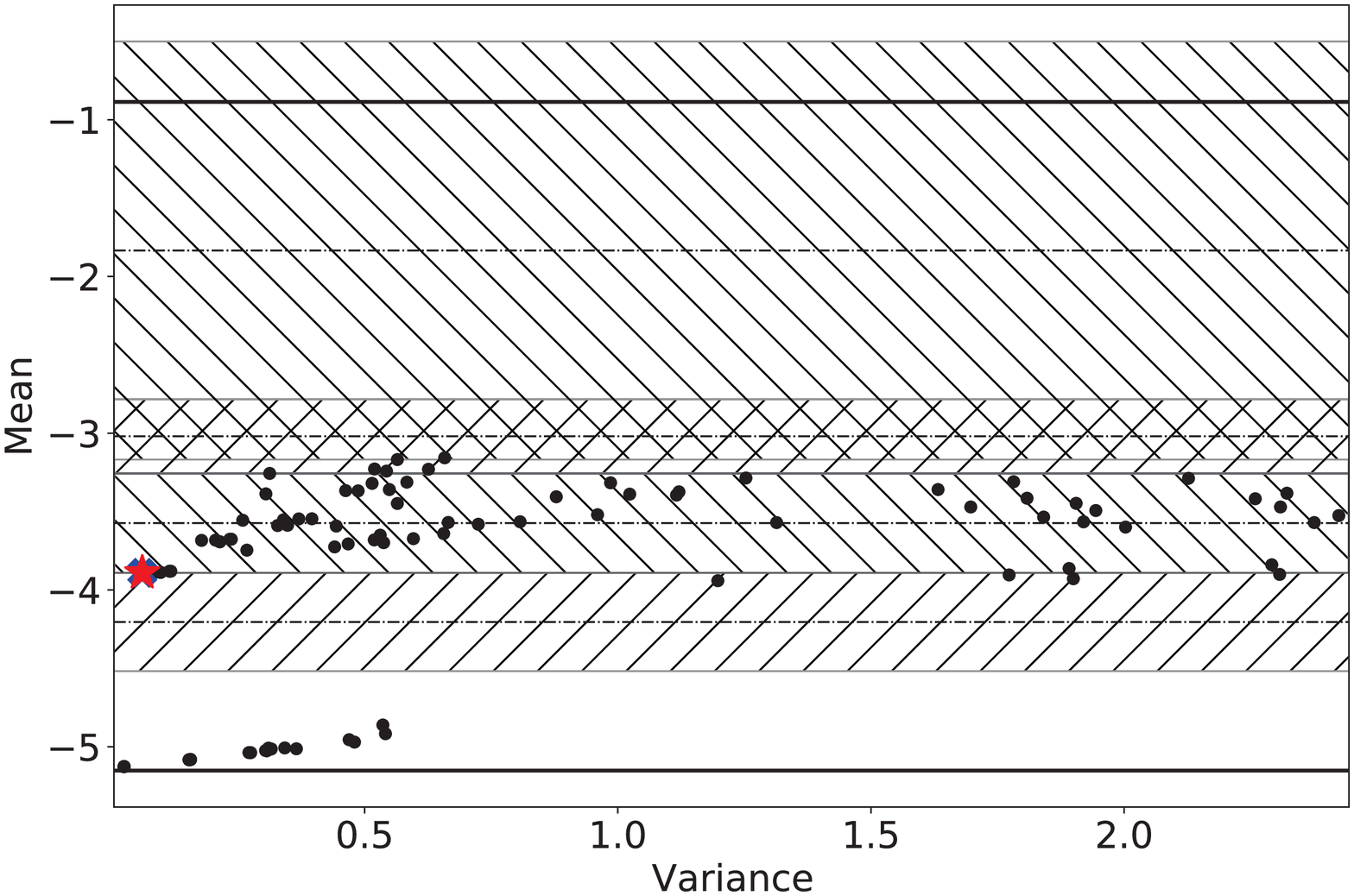}
        %       \caption{Step 4}
    \end{minipage}

    \begin{minipage}{0.49\linewidth}
        \centering
        \includegraphics[width=1\linewidth]{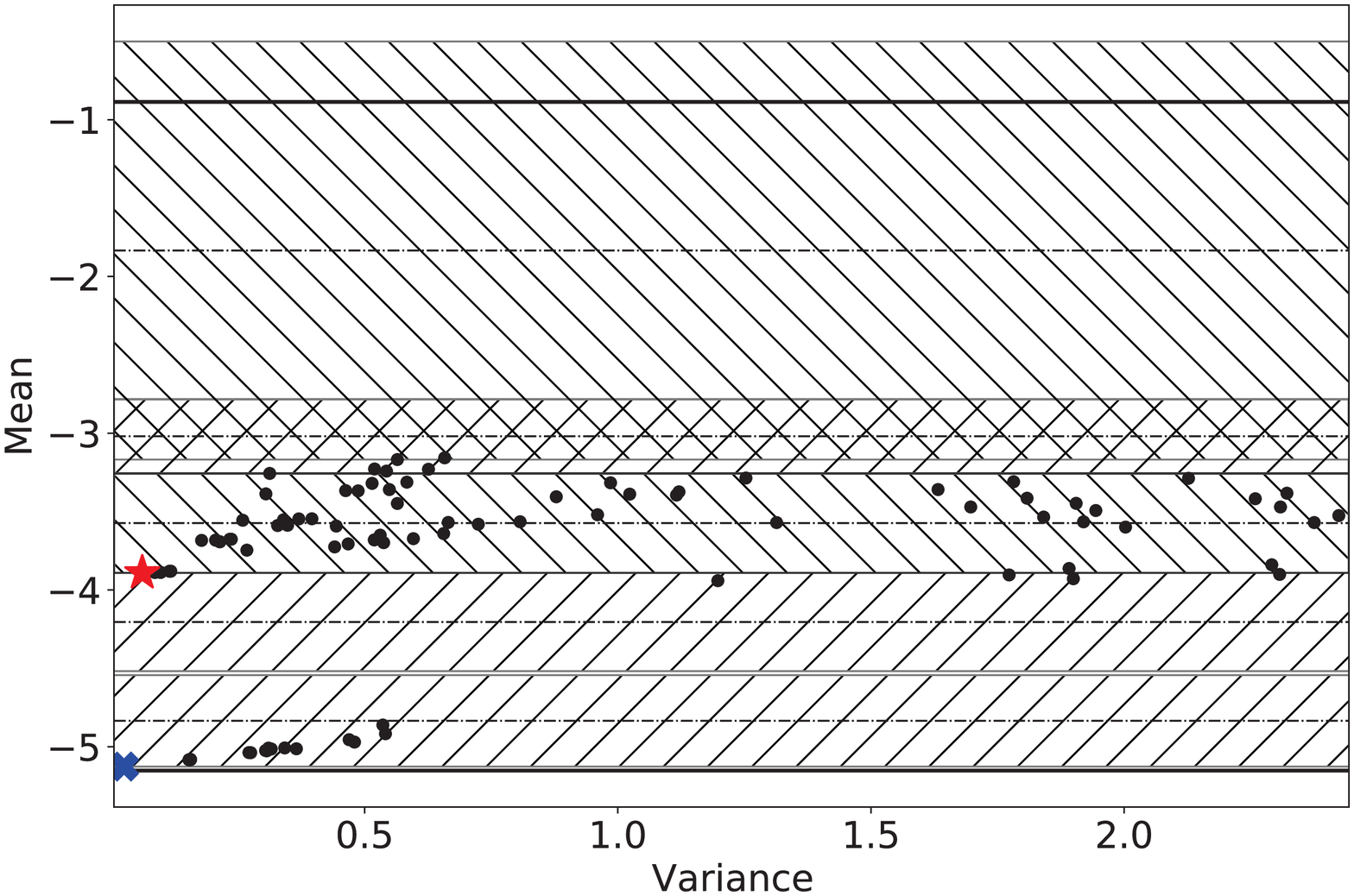}
        %       \caption{Step 5}
    \end{minipage}
    \begin{minipage}{0.49\linewidth}
        \centering
        \includegraphics[width=1\linewidth]{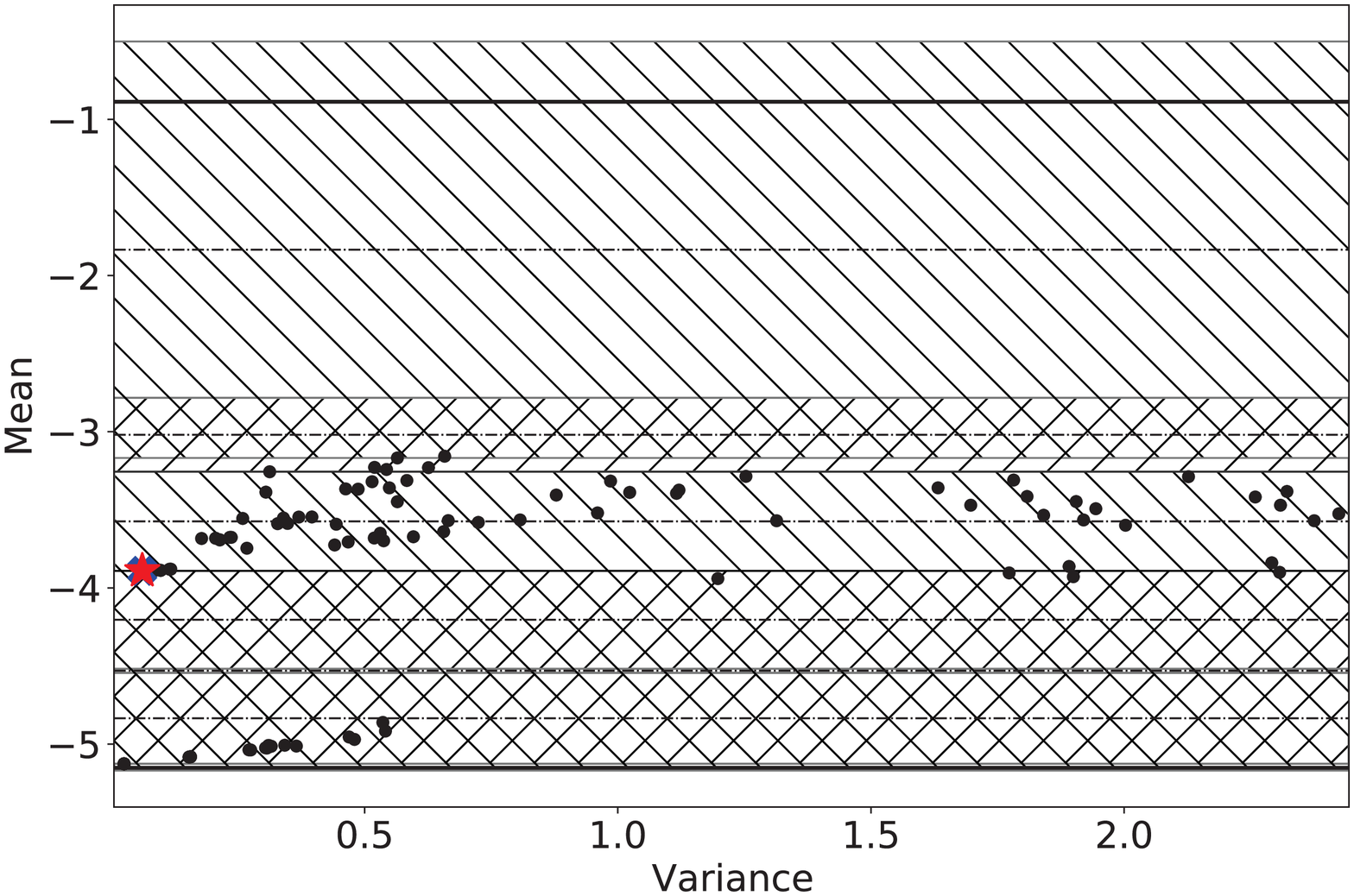}
        %       \caption{Step 6}
    \end{minipage}

    \caption{A convergence process of Algorithm~\ref{algo1}, where the blue cross signs represent the optimal policies of auxiliary problems $\mathcal M(y)$ and the red stars represent the best-so-far policies.}\label{fig_convProc}
\end{figure}

As a comparison, we also implement the local optimization algorithm
proposed by \cite{Xia20}. Considering that the mean-variance
optimization of this problem usually has multiple local optima, we
illustrate the performance comparison of these two algorithms in
Fig.~\ref{fig_compGlobalLocal}, where different problem sizes $C \in
\{4, 7, 10, 20, 30, 50 \}$ are used. We can see that our global
algorithm has much better performance and the local algorithm by
\cite{Xia20} may converge to different local optima shown by the
whiskers of standard deviations.

\begin{figure}[htbp]
\centering
\includegraphics[width=0.6\linewidth]{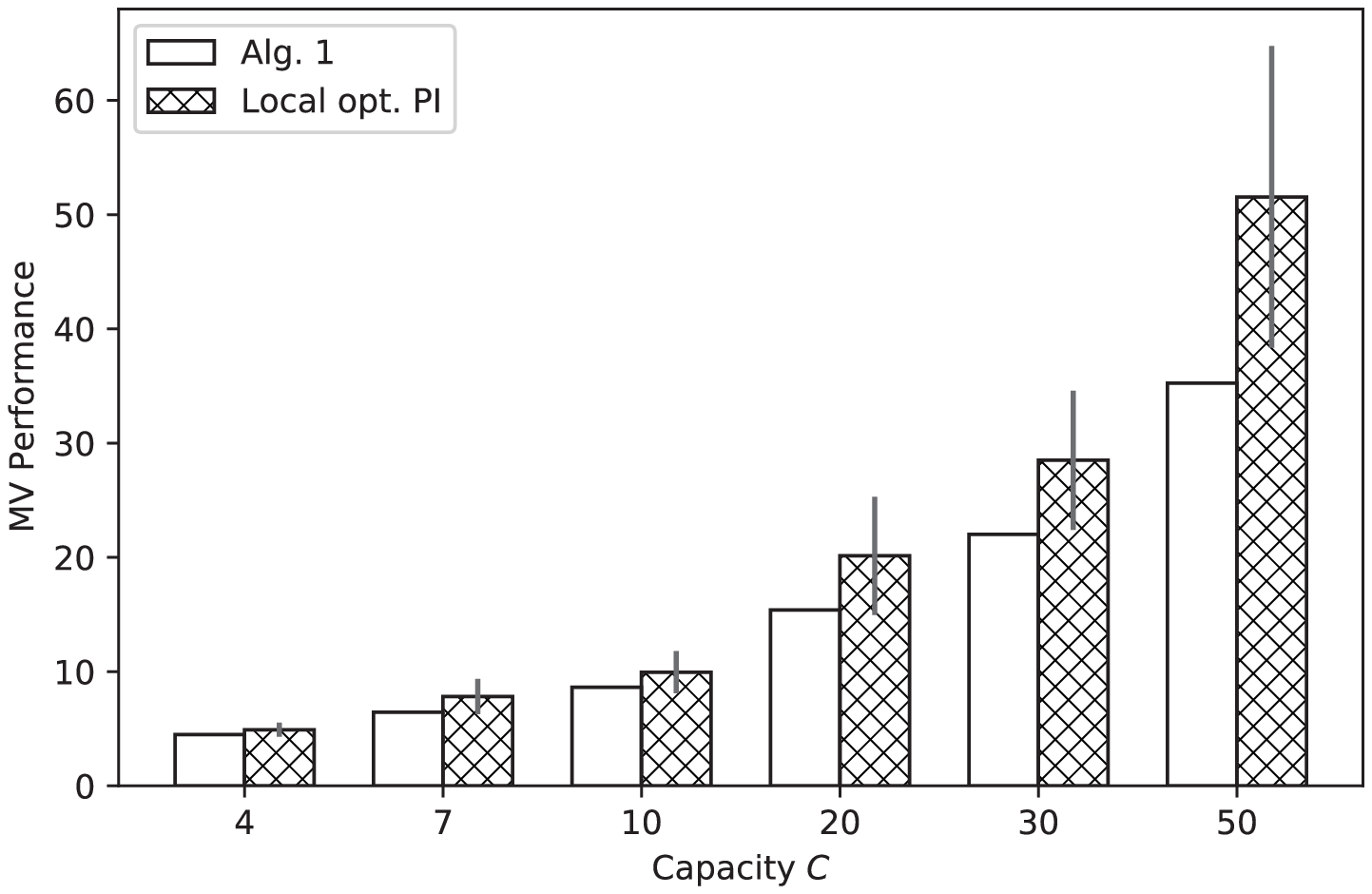}
\caption{Performance comparisons between Algorithm~\ref{algo1} and
the local optimization algorithm by~\cite{Xia20}, with different
inventory capacities.}\label{fig_compGlobalLocal}
\end{figure}

Fig.~\ref{fig_effect_capacity} shows the curves of optimal pseudo
mean-variance $\tilde{\eta}^*(y)$ with respect to the pseudo mean
$y$. For capacity $C=4$, the global optimum is $\eta^*=4.500$ and
the other two local optima are 5.376 and 6.382, which coincide with
the left pair of bars in Fig.~\ref{fig_compGlobalLocal}. The pseudo
mean corresponding to $\eta^*$ is $y^*=-3.891$, which also equals
the mean of the star point in the last subfigure of
Fig.~\ref{fig_convProc}. All these demonstrate that our
Algorithm~\ref{algo1} truly finds the global optimum and the local
algorithm by \cite{Xia20} randomly converges to different local
optima. Moreover, when the capacity increases, the curve of
$\tilde{\eta}^*(y)$ has more local optima and the local algorithm is
more possibly trapped in a worse local optimum. This also explains
the big performance gaps in Fig.~\ref{fig_compGlobalLocal} when the
capacity is large.

\begin{figure}[htbp]
    \centering
    \begin{minipage}{0.49\linewidth}
        \centering
        \includegraphics[width=1\linewidth]{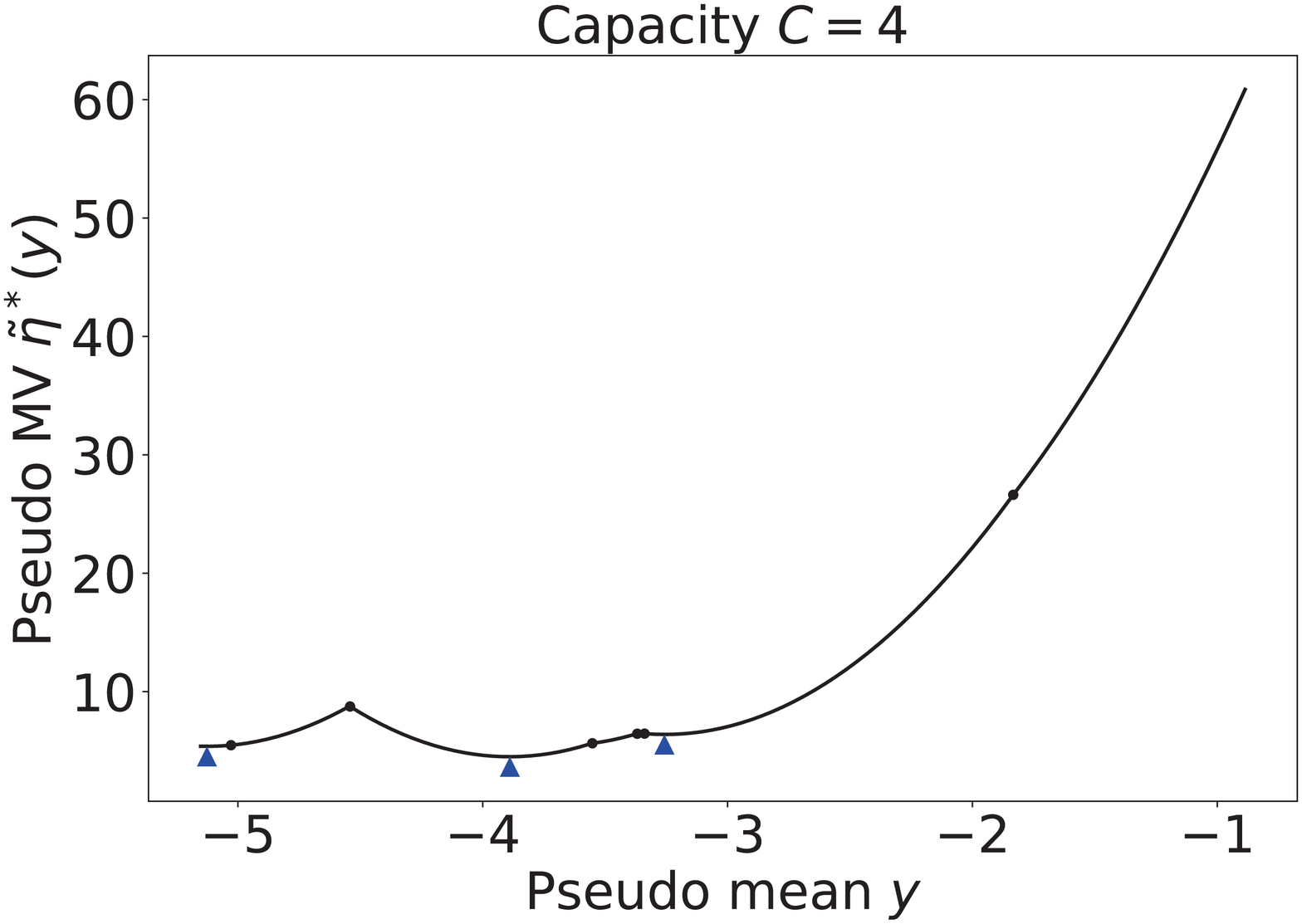}
        %       \caption{Step 1}
    \end{minipage}
    \begin{minipage}{0.49\linewidth}
        \centering
        \includegraphics[width=1\linewidth]{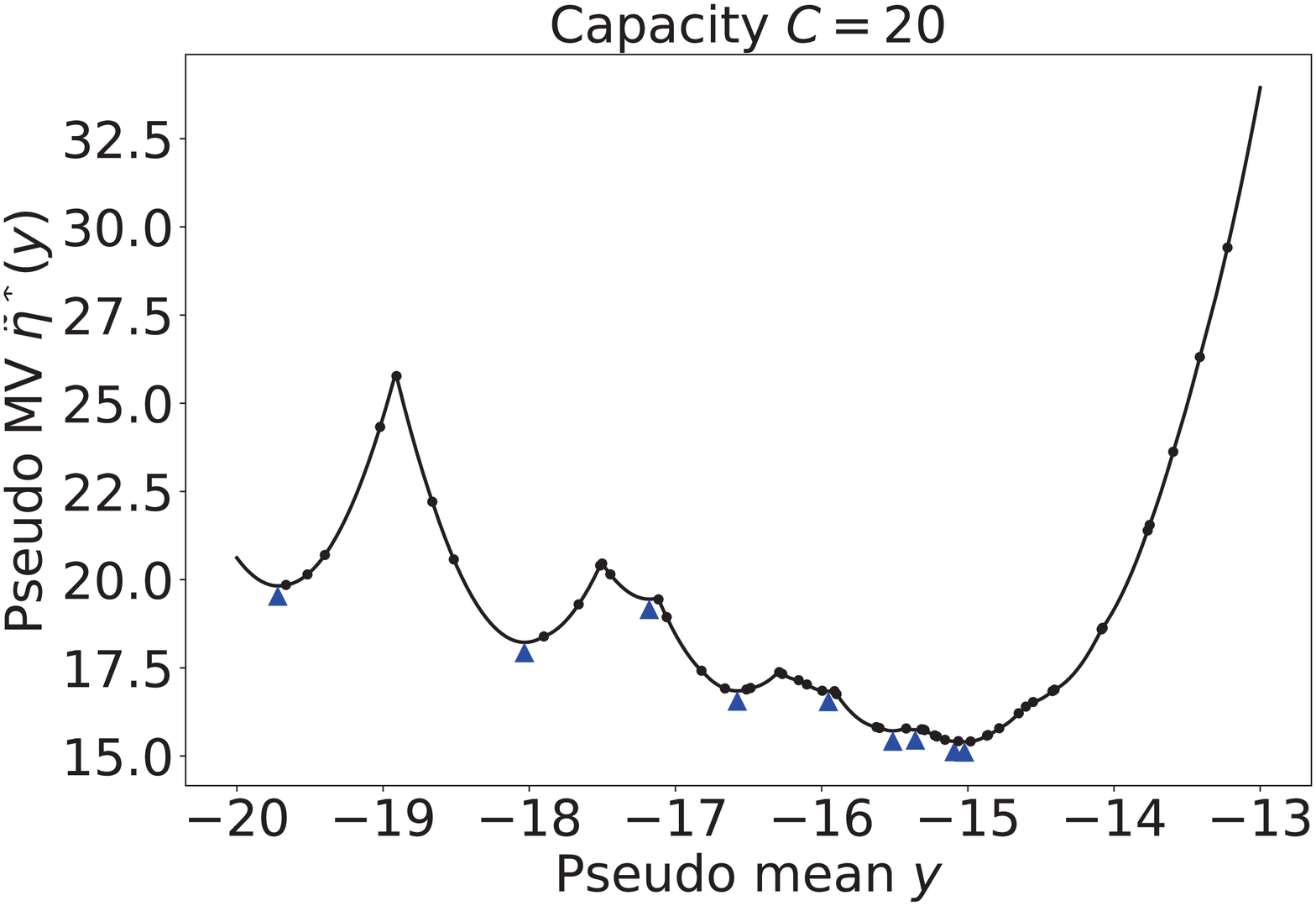}
        %       \caption{Step 5}
    \end{minipage}
    \caption{Curves of the optimal pseudo mean-variance $\tilde{\eta}^*(y)$ under different capacities.}\label{fig_effect_capacity}
\end{figure}

Furthermore, we study the effect of risk coefficient $\beta$ on the
curve $\tilde{\eta}^*(y)$, as illustrated in
Fig.~\ref{fig_pseudoMeanMVbeta}. We observe that the problem
complexity is increasing with respect to $\beta$. When $\beta$ is
small, the curve has only a single local optimum, which indicates
that the problem is easy to solve. This is because a mean-variance
problem with a small $\beta$ is approximately equivalent to only
optimizing the mean performance, which is a standard MDP easy to
solve. Oppositely, when $\beta$ is large, the curve has multiple
local optima and the associated optimization problem is difficult to
solve.

\begin{figure}[htbp]
    \centering
    \begin{minipage}{0.49\linewidth}
        \centering
        \includegraphics[width=1\linewidth]{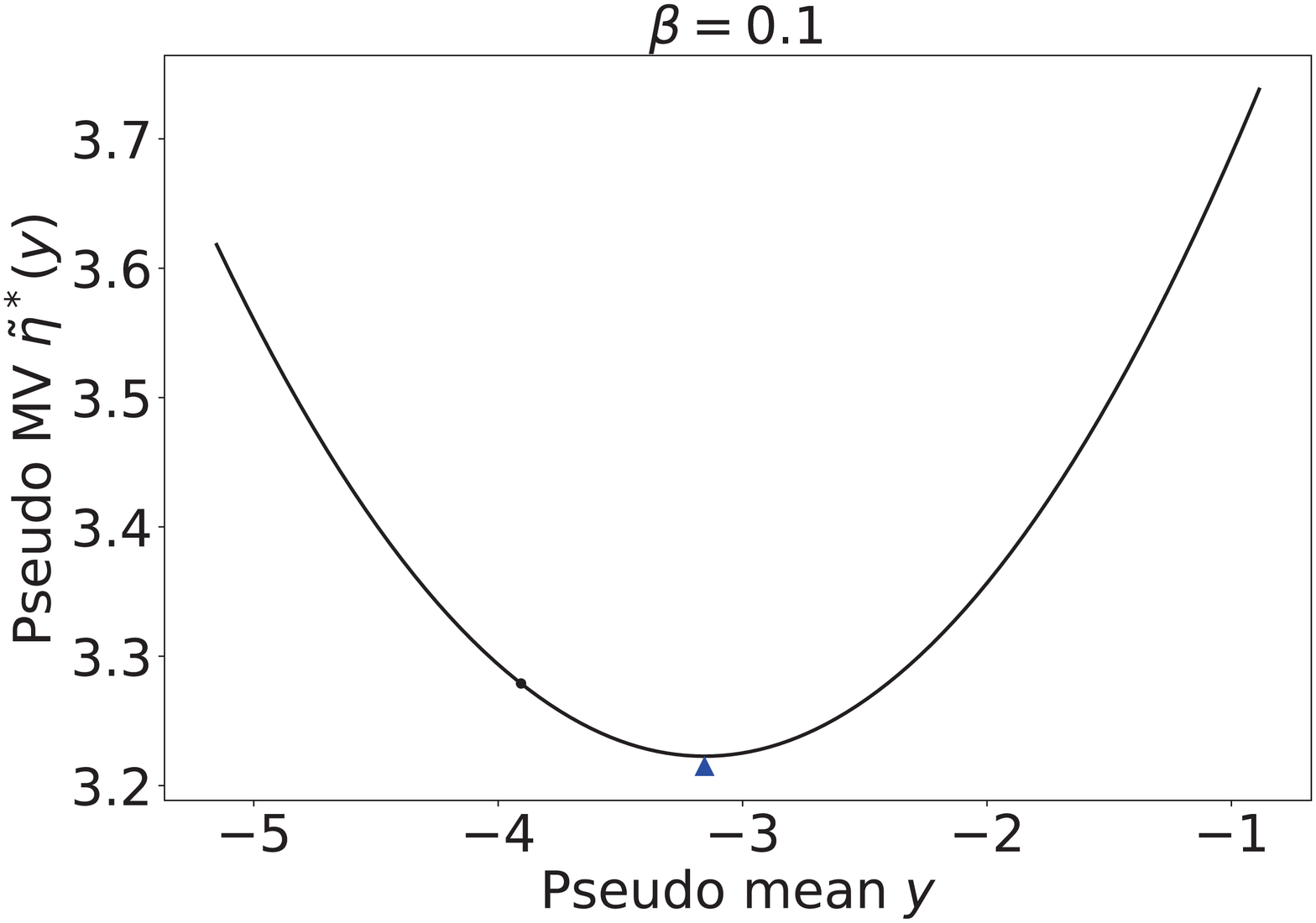}
        %       \caption{Step 2}
    \end{minipage}
    \begin{minipage}{0.49\linewidth}
        \centering
        \includegraphics[width=1\linewidth]{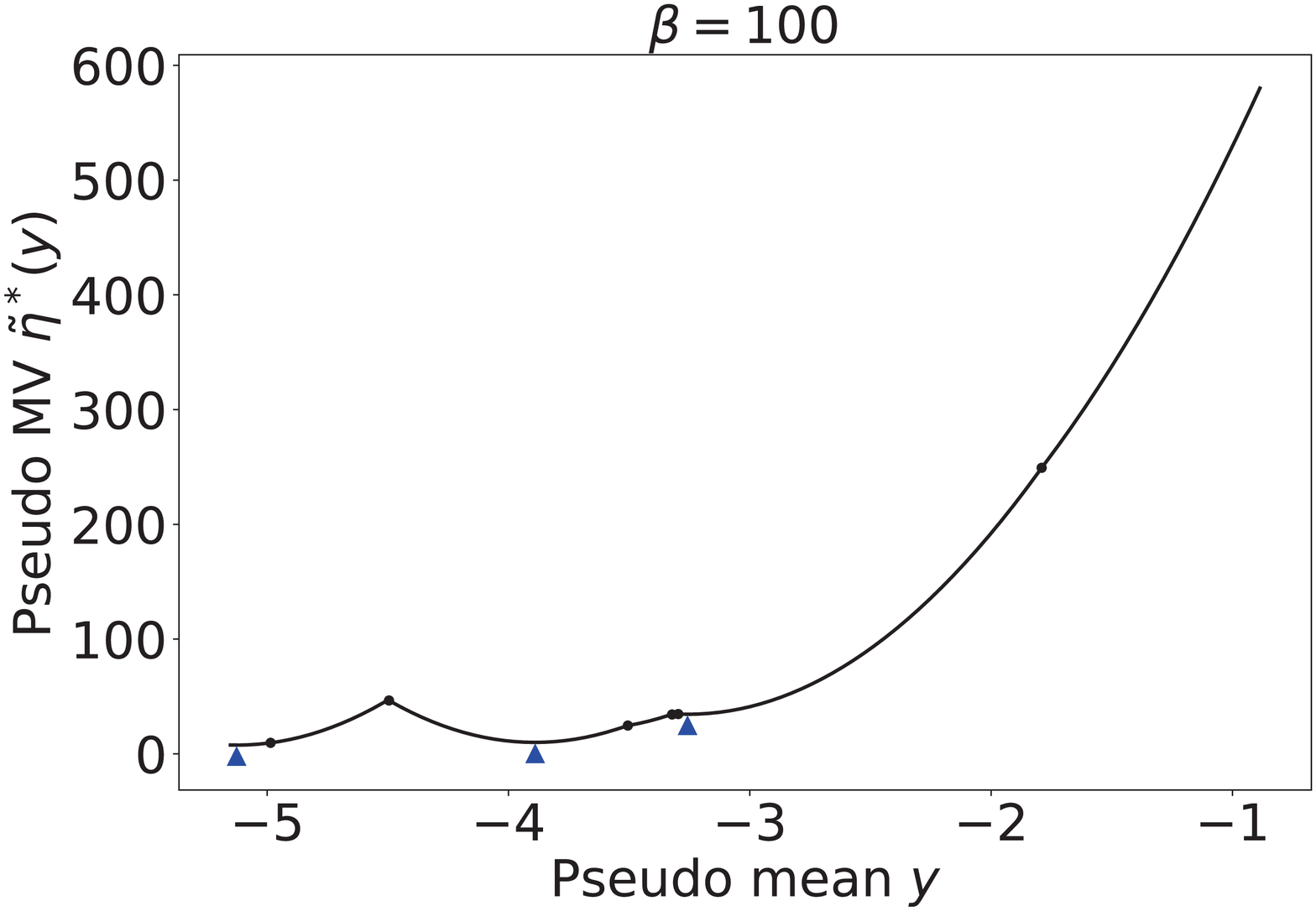}
        %       \caption{Step 5}
    \end{minipage}
   \caption{Curves of the optimal pseudo mean-variance $\tilde{\eta}^*(y)$ under different $\beta$'s.}\label{fig_pseudoMeanMVbeta}
\end{figure}

Finally, we study the effect of Lemma~\ref{lemma4} on the algorithm
efficiency. We compare the performance difference between
Algorithm~\ref{algo1} and Algorithm~\ref{algo1}-Plus under different
capacities and $\beta$'s. We observe that Algorithm~\ref{algo1}-Plus
can achieve a significant efficiency improvement when the problem
size (capacity) is large, as shown in
Fig.~\ref{fig_Lemma_MDPsolved}. When $\beta$ is changed, there are
three cases as shown in Fig.~\ref{fig_Lemma_beta_MDPsolved}:
\begin{enumerate}
    \item When $\beta$ is relatively small ($ \leq 0.1 $), the variance is trivial, and the mean-variance optimization is approximately equivalent to a mean optimization problem which is a standard MDP.
    The problem is relatively easy, and these two algorithms have similar efficiency;
    \item When $\beta$ is relatively large ($ \geq 100 $), Lemma~\ref{lemma4} may rarely remove areas with means smaller than $\underline{r}$, which can be illustrated by the intercept in Fig.~\ref{fig_dominatedarea2} when the line slope is large. Thus, these two algorithms also have similar efficiency in this case;
    \item In other cases, Lemma~\ref{lemma4} significantly improves the algorithm convergence speed, and Algorithm~\ref{algo1}-Plus is quite more efficient than Algorithm~\ref{algo1}.
\end{enumerate}

\begin{figure}[htbp]
    \centering
    \subfigure[Effect of Lemma~\ref{lemma4} under different capacities.]
    {\begin{minipage}{0.45\linewidth}
            \centering
            \includegraphics[width=1\linewidth]{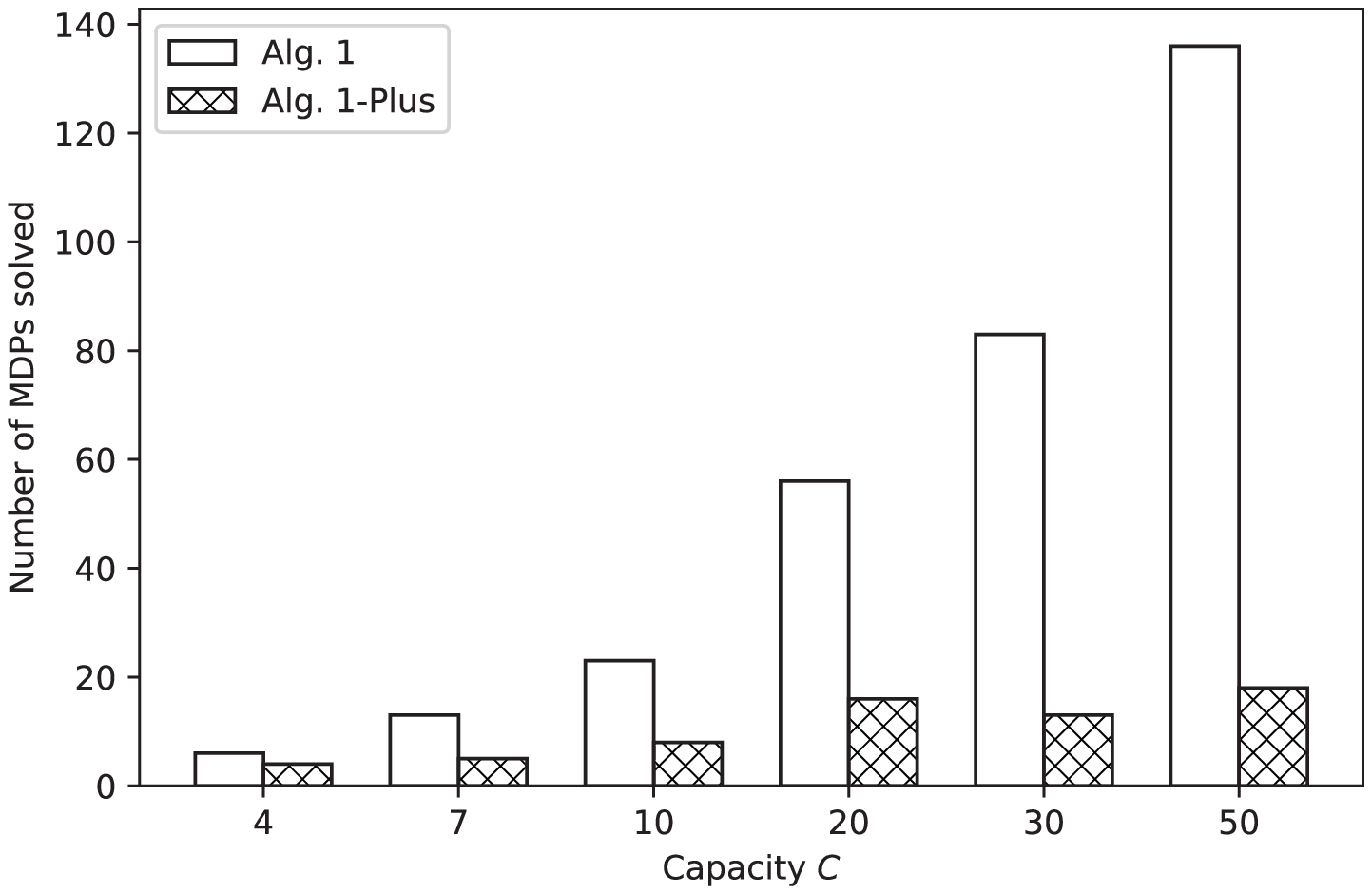}\label{fig_Lemma_MDPsolved}
        \end{minipage}
    }
    \subfigure[Effect of Lemma~\ref{lemma4} under different $\beta$'s.]
    {\begin{minipage}{0.45\linewidth}
            \centering
            \includegraphics[width=1\linewidth]{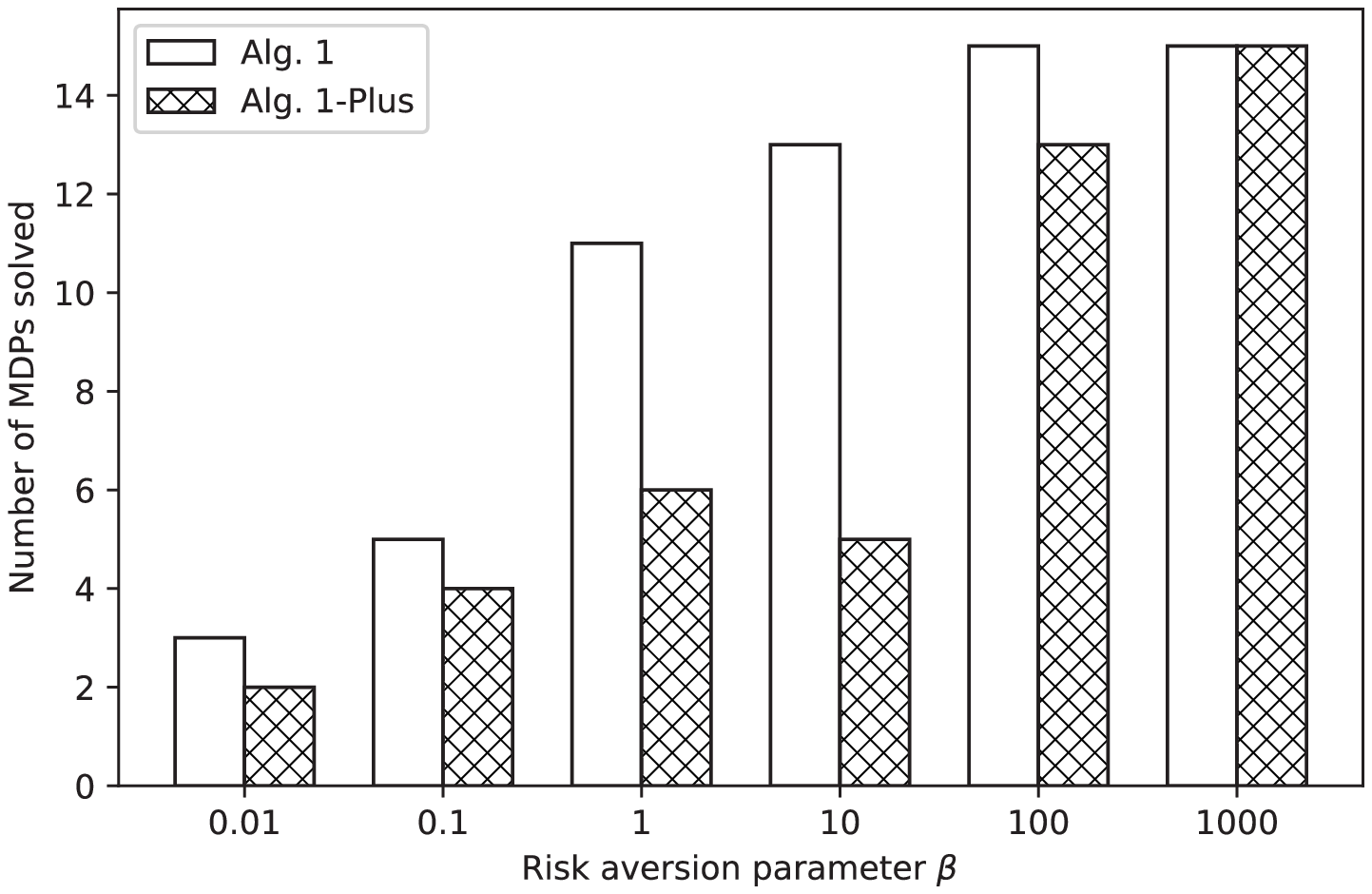}\label{fig_Lemma_beta_MDPsolved}
        \end{minipage}
    }
    \caption{Efficiency comparisons between Algorithm~\ref{algo1} and Algorithm~\ref{algo1}-Plus under different capacities and $\beta$'s.}\label{fig_Lemma4}
\end{figure}

%\subsection{Supplementary: effect of reward scaling parameter}
%For MDP with $ r' = k r $, a comparison similar to Section~\ref{section_exp_beta} is shown in Fig.~\ref{fig_comparison_k}.
%Since both problems are equivalent when $ \beta' k = \beta $, the result is similar to that in Section~\ref{section_exp_beta}.
%
%\begin{figure}[htbp]
%   \centering
%   \includegraphics[width=.6\columnwidth]{fig/comparison_k}
%   \caption{The effect of the scaling parameter $ k $ on algorithm efficiency.}\label{fig_comparison_k}
%\end{figure}

\section{Discussion and Conclusion} \label{section_conclusion}
This paper proposes the global algorithms for solving multi-period
mean-variance optimization in the framework of MDPs, which is a
long-standing challenge caused by the failure of dynamic
programming. We convert this problem to a bilevel MDP formulation,
where the inner optimization is a standard MDP $\mathcal M(y)$ for
pseudo mean-variance optimization and the outer one is a single
parameter selection problem optimizing pseudo mean $y$.
Interestingly, the optimal value of $\mathcal M(y)$ is a convex
piecewise quadratic function of $y$. By the square form difference
between the real variance and the pseudo variance, we discover
policy dominance properties to help remove worse policy spaces
iteratively. The global optimum can be found by repeatedly removing
these dominated policy spaces. The convergence and efficiency of our
algorithms are studied both theoretically and experimentally.

Our work demonstrates a promising approach to globally optimize the
steady-state mean-variance metrics in undiscounted MDPs. It is
meaningful to further extend our approach to mean-variance
optimization of discounted MDPs. Another interesting topic is to
develop reinforcement learning algorithms based on our global
optimization approach, which can make our approach implementable in
a data-driven environment.

%\section*{Acknowledgments}
%This work was supported in part by the National Natural Science
%Foundation of China (62073346, 11931018, U1811462), the Guangdong
%Basic and Applied Basic Research Foundation (2021A1515011984), the
%Guangdong Province Key Laboratory of Computational Science at the
%Sun Yat-sen University (2020B1212060032), and the Fundamental
%Research Funds for the Central Universities, Sun Yat-sen University.

\end{document}